\documentclass[11pt]{amsart}
\usepackage{graphicx}
\usepackage{amssymb,amsthm}
\usepackage{amsmath}
\usepackage[nosumlimits]{mathtools}
\usepackage[margin=1in]{geometry}
\usepackage{enumitem}
\usepackage{multirow}
\usepackage{algorithm,algorithmic}

\theoremstyle{plain}
\newtheorem{theorem}{Theorem}[section]
\newtheorem{proposition}[theorem]{Proposition}
\newtheorem{lemma}[theorem]{Lemma}
\newtheorem{corollary}[theorem]{Corollary}
\newtheorem{open}[theorem]{Open Problem}

\theoremstyle{definition}

\theoremstyle{remark}

\usepackage{hyperref}
\usepackage{xcolor}
\hypersetup{
    colorlinks,
    linkcolor={red!50!black},
    citecolor={blue!50!black},
    urlcolor={blue!80!black}
}

\newcommand{\ii}{{\ensuremath{\mathrm{i}}}}
\newcommand{\tp}{{\scriptscriptstyle\mathsf{T}}}
\newcommand{\h}{{\scriptscriptstyle\mathsf{H}}}
\newcommand{\m}{{\scriptscriptstyle-}}
\newcommand{\p}{{\scriptscriptstyle+}}

\newcommand*{\medcup}{\mathbin{\scalebox{1.2}{\ensuremath{\cup}}}}%

\let\O\undefine
\DeclareMathOperator{\O}{O}

\DeclareMathOperator{\U}{U}
\DeclareMathOperator{\SO}{SO}
\DeclareMathOperator{\SU}{SU}
\DeclareMathOperator{\Sp}{Sp}
\DeclareMathOperator{\Spin}{Spin}
\DeclareMathOperator{\SL}{SL}
\DeclareMathOperator{\GL}{GL}
\DeclareMathOperator{\Gr}{Gr}
\DeclareMathOperator{\Flag}{Flag}
\DeclareMathOperator{\tr}{tr}

\DeclareMathOperator{\diag}{diag}

\DeclareMathOperator{\spn}{span}
\DeclareMathOperator{\Orb}{Orb}
\DeclareMathOperator{\Stab}{Stab}

\makeatletter
\newenvironment{breakablealgorithm}
{% \begin{breakablealgorithm}
    \vspace{15pt}
	\begin{center}
		\refstepcounter{algorithm}% New algorithm
		\hrule height.8pt depth0pt \kern2pt% \@fs@pre for \@fs@ruled
		\renewcommand{\caption}[2][\relax]{% Make a new \caption
			{\raggedright\textbf{\ALG@name~\thealgorithm} ##2\par}%
			\ifx\relax##1\relax % #1 is \relax
			\addcontentsline{loa}{algorithm}{\protect\numberline{\thealgorithm}##2}%
			\else % #1 is not \relax
			\addcontentsline{loa}{algorithm}{\protect\numberline{\thealgorithm}##1}%
			\fi
			\kern2pt\hrule\kern2pt
		}
	}{% \end{breakablealgorithm}
		\kern2pt\hrule\relax% \@fs@post for \@fs@ruled
	\end{center}
    \vspace{15pt}
}
\makeatother

\begin{document}
\title[$\SO(n)$, $\SU(n)$, and $\Sp(2n, \mathbb{F})$ as products of Grassmannians]{Special orthogonal, special unitary, and symplectic groups as products of Grassmannians}
\author[L.-H.~Lim]{Lek-Heng~Lim}
\address{Computational and Applied Mathematics Initiative, University of Chicago, Chicago, IL 60637}
\email{lekheng@uchicago.edu}
\author[X.~Lu]{Xiang Lu}
\address{Department of Statistics, University of Chicago, Chicago, IL 60637}
\email{xiangl1@uchicago.edu}
\author[K.~Ye]{Ke~Ye}
\address{KLMM, Academy of Mathematics and Systems Science, Chinese Academy of Sciences,
Beijing 100190, China}
\email{keyk@amss.ac.cn}

\begin{abstract}
We describe a  curious structure of  the special orthogonal, special unitary, and symplectic groups that has not been observed, namely, they can be expressed as matrix products of their corresponding Grassmannians realized as involution matrices.
We will show that $\SO(n)$ is a product of two real Grassmannians, $\SU(n)$ a product of four complex Grassmannians, and $\Sp(2n, \mathbb{R})$ or $\Sp(2n, \mathbb{C})$ a product of four symplectic Grassmannians over $\mathbb{R}$ or $\mathbb{C}$ respectively. Our proofs essentially require only standard matrix theory.
\end{abstract}

\maketitle

\section{Introduction}\label{sec:intro}

By definition, the real, complex, and symplectic Grassmannians are respectively
\begin{equation}\label{eq:abstract}
\begin{aligned}
\mathbb{G}(k, \mathbb{R}^n) &= \{k\text{-dimensional real subspaces in }\mathbb{R}^n \},  \\
\mathbb{G}(k, \mathbb{C}^n) &= \{k\text{-dimensional complex subspaces in }\mathbb{C}^n \},  \\
\mathbb{G}_{\Sp}(2k, \mathbb{F}^{2n}) &= \{2k\text{-dimensional symplectic subspaces in }\mathbb{F}^{2n} \},
\end{aligned}
\end{equation}
where $\mathbb{F} = \mathbb{R}$ or $\mathbb{C}$ here and throughout.
These are real smooth manifolds and real affine varieties but the descriptions in \eqref{eq:abstract} are abstract and coordinate-free. To do almost anything with these Grassmannians, one needs a more concrete characterization of these objects. For example, a common approach is to characterize them as homogeneous spaces:
\begin{equation}\label{eq:homo}
\begin{aligned}
\mathbb{G}(k, \mathbb{R}^n) &\cong \O(n)/\bigl( \O(k) \times \O(n-k) \bigr),\\
\mathbb{G}(k, \mathbb{C}^n) &\cong \U(n)/\bigl( \U(k) \times \U(n-k) \bigr),\\
\mathbb{G}_{\Sp}(2k, \mathbb{F}^{2n}) &\cong \Sp(2n, \mathbb{F})/\bigl( \Sp(2k, \mathbb{F}) \times \Sp(2n-2k, \mathbb{F}) \bigr).
\end{aligned}
\end{equation}
In our recent works \cite{ZLK20,ZLK24curvaure,LK24degree,ZLK24,ZLK24b}, largely motivated by computational considerations, we characterized them as submanifolds of matrices:
\begin{equation}\label{eq:invo}
\begin{aligned}
\Gr(k, \mathbb{R}^{n}) &\coloneqq \{ X \in \O(n) : X^2 = I_n, \; \tr(X) = 2k-n \},\\
\Gr(k, \mathbb{C}^{n})  &\coloneqq \{ X \in \U(n) : X^2 = I_n, \; \tr(X) = 2k-n \},\\
\Gr_{\Sp}(2k, \mathbb{F}^{2n}) &\coloneqq  \{ X \in \Sp(2n,\mathbb{F}): X^2 = I_{2n},\; \tr(X) = 4k - 2n \}.
\end{aligned}
\end{equation}
These are submanifolds of $\O(n)$, $\U(n)$, $\Sp(2n,\mathbb{F})$ respectively and are isomorphic to their abstract counterparts in \eqref{eq:abstract} or homogeneous space counterparts in \eqref{eq:homo}, whether as smooth manifolds or as real affine varieties. We call them the \emph{involution models} for the  real, complex, and symplectic Grassmannians respectively. Explicit formulas for the isomorphisms in \eqref{eq:homo} and \eqref{eq:invo} may be found in \cite{ZLK20,ZLK24}. 

The goal of our article is to state and prove a somewhat curious structure. For any two subsets of $n \times n$ matrices $S_1$ and $S_2$, write $S_1 \cdot S_2 = \{X_1 X_2 : X_1 \in S_1$, $X_2 \in S_2 \}$. Then
\begin{equation}\label{eq:main}
\begin{aligned}
\SO(n) &= \Gr(\lfloor \tfrac{n}{2} \rfloor, \mathbb{R}^n) \cdot \Gr(\lfloor \tfrac{n}{2} \rfloor,\mathbb{R}^n), \\
\SU(n) &= \Gr(\lfloor \tfrac{n}{2} \rfloor, \mathbb{C}^n) \cdot \Gr(\lfloor \tfrac{n}{2} \rfloor,\mathbb{C}^n) \cdot \Gr(\lfloor \tfrac{n}{2} \rfloor, \mathbb{C}^n) \cdot \Gr(\lfloor \tfrac{n}{2} \rfloor,\mathbb{C}^n), \\
\Sp(2n,\mathbb{F}) &= \Gr_{\Sp}(2\lfloor \tfrac{n}{2} \rfloor, \mathbb{F}^{2n}) \cdot \Gr_{\Sp}(2\lfloor \tfrac{n}{2} \rfloor,\mathbb{F}^{2n}) 
\cdot \Gr_{\Sp}(2\lfloor \tfrac{n}{2} \rfloor,\mathbb{F}^{2n}) \cdot \Gr_{\Sp}(2\lfloor \tfrac{n}{2} \rfloor, \mathbb{F}^{2n}). 
\end{aligned}
\end{equation}
To the best of our knowledge, this neat relation between a classical group and its corresponding Grassmannian has never been observed before, which is somewhat surprising given that these groups and Grassmannians are ubiquitous and have been thoroughly studied.

It is well-known that isometries are compositions of involutions, whether over real \cite{hoffman_products_1971} or complex \cite{gustafson_products_1976, halmos_products_1958, radjavi_products_1968}, or in a symplectic setting \cite{Ad20,dR15}; see \cite{gustafson_products_1991} for a general overview. We stress that our results in this article cannot be deduced from these existing results and our proofs are of a distinct nature from those in \cite{Ad20, dR15, gustafson_products_1991, gustafson_products_1976, halmos_products_1958, hoffman_products_1971, radjavi_products_1968}. On the other hand, over finite-dimensional spaces, the key results in \cite{Ad20, dR15, gustafson_products_1991, gustafson_products_1976, halmos_products_1958, hoffman_products_1971, radjavi_products_1968} will follow from ours.

Our results in \eqref{eq:main} are exhaustive in an appropriate sense: For a real or complex vector space $\mathbb{V}$ equipped with an additional structure, the Grassmannian over $\mathbb{V}$ should generally respect the group action that preserves that structure. The four cases in \eqref{eq:abstract} discussed in this article cover the most common type of structures on $\mathbb{V}$ --- a nondegenerate bilinear form $\beta$. For a positive definite symmetric or Hermitian $\beta$, the group action is given by $\O(n)$ or $\U(n)$ and we obtain $\mathbb{G}(k, \mathbb{R}^n)$ or $\mathbb{G}(k, \mathbb{C}^n)$. For a skew-symmetric $\beta$, the group action is given by $\Sp(2n, \mathbb{F})$ and we obtain $\mathbb{G}_{\Sp}(2k, \mathbb{F}^{2n})$. We leave other, more esoteric, possibilities to future work.

We prove our results almost entirely from scratch using standard matrix theory,  developing various specialized canonical forms resembling ``eigenvalue decomposition'' for an orthogonal matrix with special symmetry and trace constraint (Lemma~\ref{lem:char}),  a product of two orthogonal involution matrices (Theorem~\ref{thm:char}) and its unitary analogue (Theorem~\ref{thm:char_C_M2}), and a symplectic involution matrix (Lemma~\ref{lem:SPidentification}).  There does not seem to be any prior work that we could draw upon. In particular, none of the techniques in \cite{Ad20, dR15, gustafson_products_1991, gustafson_products_1976, halmos_products_1958, hoffman_products_1971, radjavi_products_1968} could be adapted to meet our requirements.

A few words for readers familiar with Lie Theory: We view the decompositions in \eqref{eq:main} as similar to the classical Lie group decompositions of Bruhat,  Cartan,  Iwasawa,  Jordan, Langlands,  Levi, et al \cite{Borel91, Helgason01, Bump13}. Like these classical decompositions, our product $\cdot$ is given by matrix product; but while the factors in these classical decompositions form subgroups, our factors form subvarieties. Another noteworthy point is that the involution models of Grassmannians in \eqref{eq:invo} are adjoint orbits but the equalities in \eqref{eq:main} do not hold for general adjoint orbits. In fact they do not even hold when $\lfloor n/2 \rfloor$ is replaced with other values. For example we will see in Theorem~\ref{thm:prod-dim} that $\SO(n) \ne \Gr(k_1,\mathbb{R}^n) \cdot \Gr(k_2,\mathbb{R}^n)$ for almost all other values of $k_1,k_2$.

Lastly,  even though we have relied on the involution models in \eqref{eq:invo} to establish \eqref{eq:main}, these results are nevertheless ``model independent'' in an appropriate  sense---they hold true for any arbitrary representations of these groups. Let $\rho : \SO(n) \to \GL(\mathbb{U})$, $\widehat{\rho} : \SU(n)  \to \GL(\mathbb{V})$,  $\widetilde{\rho} : \Sp(2n,\mathbb{F}) \to \GL(\mathbb{W})$ be any representations of the respective groups.  Then it follows from \eqref{eq:main} that
\begin{align*}
\rho\bigl( \SO(n) \bigr) &=\rho\bigl(\Gr(\lfloor \tfrac{n}{2} \rfloor ,  \mathbb{R}^n) \bigr)  \rho \bigl( \Gr(\lfloor \tfrac{n}{2} \rfloor,\mathbb{R}^n) \bigr), \\
\widehat{\rho} \bigl( \SU(n) \bigr) &=\widehat{\rho}\bigl( \Gr(\lfloor \tfrac{n}{2} \rfloor, \mathbb{C}^n) \bigr)  \widehat{\rho}\bigl( \Gr(\lfloor \tfrac{n}{2} \rfloor,\mathbb{C}^n) \bigr)  \widehat{\rho}\bigl(\Gr(\lfloor \tfrac{n}{2} \rfloor, \mathbb{C}^n) \bigr)  \widehat{\rho}\bigl(\Gr(\lfloor \tfrac{n}{2} \rfloor,\mathbb{C}^n) \bigr), \\
 \widetilde{\rho}\bigl(\Sp(2n,\mathbb{F}) \bigr)&=  \widetilde{\rho}\bigl(\Gr_{\Sp}(2\lfloor \tfrac{n}{2} \rfloor, \mathbb{F}^{2n}) \bigr)   \widetilde{\rho}\bigl(\Gr_{\Sp}(2\lfloor \tfrac{n}{2} \rfloor,\mathbb{F}^{2n}) \bigr)   \widetilde{\rho}\bigl(\Gr_{\Sp}(2\lfloor \tfrac{n}{2} \rfloor,\mathbb{F}^{2n}) \bigr)   \widetilde{\rho}\bigl(\Gr_{\Sp}(2\lfloor \tfrac{n}{2} \rfloor, \mathbb{F}^{2n}) \bigr).
\end{align*}

\subsection{Notations and terminologies}\label{sec:nota}

We write $\mathbb{N}$ for the positive integers. Elements of $\mathbb{R}^n$ or $\mathbb{C}^n$ delimited by parentheses like $(a_1,\dots,a_n)$ will be taken as \emph{column} vectors.

We write $\mathsf{S}^2(\mathbb{R}^n)$ for the space of real symmetric matrices and $\mathsf{\Lambda}^2(\mathbb{R}^n)$ for the space of real skew-symmetric matrices. We denote the spectrum of any $X\in\mathbb{C}^{n \times n}$ by $\sigma(X)$.  We write $I$ for an identity matrix of appropriate dimension or $I_n$ when the dimension needs to be specified; likewise we write $0_{m,n}$ for an $m \times n$ zero matrix or just  $0$ when its dimensions are clear from context. We will also define the special matrices
\begin{gather*} 
I_{m,n} = \begin{bmatrix} I_m & 0 \\ 0 &  -I_n \end{bmatrix} \in \mathbb{R}^{(m + n) \times (m + n)}, \qquad
J_{2n} = \begin{bmatrix}
0 & I_n \\
-I_n & 0
\end{bmatrix} \in \mathbb{R}^{2n \times 2n},\\
K_n = \begin{bmatrix}
            0 & 0 & \cdots & 0 & 1 \\
            0 & 0 & \cdots & 1 & 0 \\
            \vdots & \vdots & \reflectbox{$\ddots$} & \vdots & \vdots \\
            0 & 1 & \cdots & 0 & 0 \\
            1 & 0 & \cdots & 0 & 0
            \end{bmatrix} \in \mathbb{R}^{n \times n}.
\end{gather*}
As usual, the orthogonal, special orthogonal, unitary, special unitary, and symplectic groups are
\begin{align*}
\O(n) &=\{ X\in \mathbb{R}^{n\times n}: X^\tp X = I \}, & \SO(n) &=\{ X \in \O(n) : \det(X)=1 \},\\
\U(n) &= \{X \in \mathbb{C}^{n \times n} : X^\h X  = I\}, & \SU(n) &=\{ X\in \U(n): \; \det(X)=1 \},\\
\Sp(2n,\mathbb{F}) &=  \{ X \in \mathbb{F}^{2n \times 2n}: X^\tp J_{2n} X = J_{2n} \},
\end{align*}
where the last may be over $\mathbb{F} = \mathbb{R}$ or $\mathbb{C}$. Other special sets of interests are $\SO^\m (n) \coloneqq \{ X \in \O(n) : \det(X)=-1 \}$ and $\SU^\m (n) \coloneqq \{ X \in \U(n) : \det(X)=-1 \}$.

Let $d, n ,k_1,  \dots,  k_d  \in \mathbb{N}$ and we will always assume that $1 \le k_1,\dots,k_d \le n -1$.  We write
\begin{equation}\label{eq:prodmap}
\begin{aligned}
\Phi(k_1,\dots,  k_d, \mathbb{R}^n) &= \{X_1\cdots  X_d  \in \O(n) : X_i \in \Gr(k_i, \mathbb{R}^n),\; i=1, \dots, d\},  \\
\Phi(k_1,\dots,  k_d, \mathbb{C}^n) &= \{X_1\cdots  X_d  \in \U(n) : X_i \in \Gr(k_i, \mathbb{C}^n),\; i=1, \dots, d\}, \\
\Phi_{\Sp}(2k_1,\dots, 2 k_d, \mathbb{F}^{2n}) &= \{X_1\cdots  X_d  \in \Sp(2n, \mathbb{F}) : X_i \in \Gr_{\Sp}(2k_i, \mathbb{F}^{2n})),\; i=1, \dots, d\}.
\end{aligned}
\end{equation}
Henceforth whenever we refer to a ``Grassmannian'' it will be in sense of the involution models in \eqref{eq:invo}; in particular, a point on a Grassmannian is a matrix. The word ``product'' would mean matrix product, in the sense of two matrices or two subsets of matrices. All subsets of $\mathbb{C}^{n \times n}$ considered in this article will be \emph{real} semialgebraic sets and whenever we speak of their dimension, it will be in the sense of \emph{real} dimension. The one exception is in the context of symplectic $\mathbb{F}$-vector spaces, where we will denote dimension by $\dim_\mathbb{F}$ to emphasize the dependence on $\mathbb{F}$.

\section{Product of two real Grassmannians}\label{sec:dim}

We will study a few properties of $\Phi(k,k',\mathbb{R}^n)$ with the goal of determining its dimension, from which we will deduce our required result that $\Phi(\lfloor \tfrac{n}{2} \rfloor, \lfloor \tfrac{n}{2} \rfloor, \mathbb{R}^n) = \SO(n)$. We remind the reader that an alternative description of the involution model is as
\begin{equation}\label{eq:equiv}
\Gr(k,\mathbb{R}^n) = \{ X \in \mathsf{S}^2(\mathbb{R}^n):  X^2 = I,\; \tr(X) = 2k - n \}
\end{equation}
since an involution is orthogonal if and only if it is symmetric.

We begin with an observation about a product of points from two Grassmannians.
\begin{lemma}\label{lem:conditions}
Let $Z \in \O(n)$ and $k, k' \in \mathbb{N}$. Then the following are equivalent:
\begin{enumerate}[label={\upshape(\roman*)}, ref=\roman*] % [\upshape (i)]
    \item $Z \in \Phi(k,k',\mathbb{R}^n)$.  \label{lem:conditions:item1}
    \item $Q Z Q^\tp  \in \Phi(k,k', \mathbb{R}^n)$ for all $Q\in \O(n)$. \label{lem:conditions:item3}
    \item $Y Z Y = Z^\tp$ and $\tr(ZY) = 2k - n$ for some $Y\in \Gr(k',\mathbb{R}^n)$.  \label{lem:conditions:item2}
\end{enumerate}
\end{lemma}
\begin{proof}
Let $Z = X Y \in \Phi(k,k',\mathbb{R}^n)$.  Then  $Q Z Q^\tp = (Q X Q^\tp) (Q Y Q^\tp) \in \Phi(k,k')$ for all $Q\in \O(n)$ and we have \eqref{lem:conditions:item1}$\Rightarrow$\eqref{lem:conditions:item3}.  Clearly \eqref{lem:conditions:item3}$\Rightarrow$\eqref{lem:conditions:item1} with $Q = I$. 

Let  $Z = X Y$ with $X\in \Gr(k, \mathbb{R}^n)$, $Y \in \Gr(k',\mathbb{R}^n)$. Then $Z Y = X  \in \Gr(k,\mathbb{R}^n)$, and so $(ZY)^2 = I$ and $\tr(ZY) = 2k - n$. Hence \eqref{lem:conditions:item1}$\Rightarrow$\eqref{lem:conditions:item2}. Conversely, if $(ZY)^2 = I$ and $\tr(ZY) = 2k - n$ for some $Y\in \Gr(k',\mathbb{R}^n)$,  then as $(ZY) (ZY)^\tp = I = (ZY)^2$, we have $X \coloneqq ZY \in \mathsf{S}^2(\mathbb{R}^n)$ and $\tr(X) = 2k - n$.  Thus $X \in \Gr(k,\mathbb{R}^n)$ by \eqref{eq:equiv}.  This shows \eqref{lem:conditions:item2}$\Rightarrow$\eqref{lem:conditions:item1}.
\end{proof}

We next deduce a canonical form for matrices satisfying $I_{k,n-k} Y I_{k,n-k} = Y^\tp$ that we need in the next result.  
\begin{lemma}\label{lem:char}
Let $k,  n \in \mathbb{N}$ with $2k \le n$.  Let $Y\in \O(n)$. Then $I_{k,n-k} Y I_{k,n-k} = Y^\tp$ and $\tr(Y I_{k,n-k}) = r$ if and only if there exist $
U \in \O(k)$, $V \in \O(n-k)$,  and diagonal matrices $\Sigma,  E_1, E_2 \in \mathbb{R}^{k \times k}$, and $E_3 \in \mathbb{R}^{(n-2k) \times (n-2k)}$ where the diagonal entries of $\Sigma$ are contained in $[0,1]$ and those of $E_1,E_2,E_3$ are $\pm 1$,  such that 
    \begin{equation}\label{lem:char:eq1}
    \Sigma (I_k - \Sigma) (E_1 - E_2) = 0,\qquad \tr \bigl( \sqrt{I_k - \Sigma^2} (E_1 - E_2) \bigr)  - \tr(E_3) = r
    \end{equation}
    and 
    \begin{equation}\label{lem:char:eq2}
Y = 
\begin{bmatrix}
U & 0 \\
0 & V
\end{bmatrix}
\begin{bmatrix}
\sqrt{I_k - \Sigma^2} E_1  & \Sigma & 0  \\
- \Sigma  &  \sqrt{I_k - \Sigma^2} E_2 & 0 \\
0 & 0 & E_3
\end{bmatrix}
\begin{bmatrix}
U & 0 \\
0 & V
\end{bmatrix}^\tp.
\end{equation}
\end{lemma}
\begin{proof}
Partition $Y$ as
\[
Y = \begin{bmatrix}
Y_1 & Y_2 \\
Y_3 & Y_4
\end{bmatrix},\; \; Y_1 \in \mathbb{R}^{k\times k},\; \; Y_2 \in \mathbb{R}^{k \times (n-k)},\; \; Y_3 \in \mathbb{R}^{(n-k) \times k},\; \; Y_4 \in \mathbb{R}^{(n-k)\times (n-k)}.
\]
If $I_{k,n-k} Y I_{k,n-k} = Y^\tp$,  then
\[
\begin{bmatrix}
Y_1 & -Y_2 \\
-Y_3 & Y_4
\end{bmatrix} = \begin{bmatrix}
Y_1^\tp & Y_3^\tp \\
Y_2^\tp & Y_4^\tp
\end{bmatrix}
\]
and so $Y_1 \in \mathsf{S}^2(\mathbb{R}^k)$,  $Y_4 \in \mathsf{S}^2(\mathbb{R}^{n-k})$, and $Y_2 = -Y_3^\tp$.  Since $Y^\tp Y = I$,  we get
\begin{equation}\label{lem:char:eq3}
Y_1^2 + Y_2 Y_2^\tp = I_k,  \quad -Y_1 Y_2 + Y_2 Y_4 = 0,\quad Y_2^\tp Y_2 + Y_4^2 = I_{n-k}.
\end{equation}
Let $Y_2 = U [\Sigma , 0] V^\tp$ be a singular value decomposition with $(U,V) \in \O(k) \times \O(n-k)$ and
\[
[\Sigma , 0] = \begin{bmatrix}
\sigma_1 & \cdots & 0 & 0 & \cdots & 0  \\
\vdots & \ddots & \vdots & \vdots & \ddots & \vdots  \\
0 & \cdots &  \sigma_k & 0 & \cdots & 0
\end{bmatrix}  \in \mathbb{R}^{k \times (n-k)},\quad \sigma_1 \ge \cdots \ge \sigma_k.
\]
Note that $\sigma_1, \dots , \sigma_k \in [0,1]$ as $Y_2$ is a block of an orthogonal matrix. By \eqref{lem:char:eq3},  we get
\[
Y_1^2= U (I_k - [ \Sigma, 0] [\Sigma, 0]^\tp) U^\tp,  \quad -Y_1 Y_2 + Y_2 Y_4 = 0,\quad Y_4^2 = V (I_{n-k} - [\Sigma, 0]^\tp [\Sigma, 0]) V^\tp.
\]
Write $D \coloneqq \diag\bigl( \sqrt{1-\sigma_1^2},\dots,  \sqrt{\smash[b]{1- \sigma_k^2}} \bigr) = \sqrt{I_ k - \Sigma^2}$. Then
\[
Y_1 = U D E_1 U^\tp,\quad Y_4 = V \begin{bmatrix}
DE_2 & 0 \\
0 & E_3
\end{bmatrix} V^\tp, 
\]
where $E_1,E_2 \in \mathbb{R}^{k \times k}$, and $E_3 \in \mathbb{R}^{(n-2k) \times (n-2k)}$ are diagonal matrices with diagonal entries $\pm 1$ satisfying $\Sigma D(E_1 - E_2) = 0$. Therefore
\[
Y = 
\begin{bmatrix}
U & 0 \\
0 & V
\end{bmatrix}
\begin{bmatrix}
D E_1  & \Sigma & 0  \\
- \Sigma  &  DE_2 & 0 \\
0 & 0 & E_3
\end{bmatrix}
\begin{bmatrix}
U & 0 \\
0 & V
\end{bmatrix}^\tp
\]
as required.
\end{proof}
Although we prefer to rely on the more familiar singular value decomposition in our proof, we thank our reviewer for pointing out that Lemma~\ref{lem:char} may also be obtained by applying the CS decomposition \cite[Section~2.6]{GV} to $Y$. Indeed, the decomposition in \eqref{lem:char:eq2} contains features of both the CS decomposition as well as the real Schur decomposition \cite[Section~B.5]{Higham}. We may now characterize a point in $\Phi(k,k',\mathbb{R}^n)$.
\begin{theorem}[Canonical form for product of two real Grassmannians]\label{thm:char}
Let $k, k', n \in \mathbb{N}$ with $k' \le k < n$ and $k + k' \le n$. Then  $Z  \in \Phi(k,  k', \mathbb{R}^n)$ if and only if there exist $Q  \in \O(n)$ and diagonal matrices $\Sigma,  E \in \mathbb{R}^{k' \times k'}$ with diagonal entries in $[0,1]$ and $\{-1,1\}$ respectively,  such that 
\[
Z = 
Q \begin{bmatrix}
\sqrt{I_{k'} - \Sigma^2} E  & \Sigma & 0  \\
- \Sigma  &  \sqrt{I_{k'} - \Sigma^2} E & 0 \\
0 & 0 & I_{n-k - k',  k - k'}
\end{bmatrix} Q^\tp.
\]
\end{theorem}
\begin{proof}
By Lemma~\ref{lem:conditions},  $Z \in \Phi(k, k', \mathbb{R}^n)$ if and only if $X Z X = Z^\tp$ and $\tr(ZX) = 2k - n$ for some $X = R I_{k',  n-k'} R^\tp \in \Gr(k', \mathbb{R}^n)$,  $R\in \O(n)$.  Let $Y \coloneqq R^\tp Z R$.  Then by Lemma~\ref{lem:char}, $Y$ decomposes as in \eqref{lem:char:eq2} into factors satisfying \eqref{lem:char:eq1}.  Setting $Q = R \diag (U,  V)$,  we obtain 
\begin{equation}\label{eq:Z1}
Z = 
Q \begin{bmatrix}
\sqrt{I_{k'} - \Sigma^2} E_1  & \Sigma & 0  \\
- \Sigma  &  \sqrt{I_{k'} - \Sigma^2} E_2 & 0 \\
0 & 0 & E_3
\end{bmatrix} Q^\tp
\end{equation}
where
\[
\Sigma (I_{k'} - \Sigma) (E_1 - E_2) = 0,\quad \tr \big( \sqrt{I_{k'} - \Sigma^2} (E_1 - E_2) \bigr)  - \tr(E_3) = 2k - n.
\]
We may further assume that $E_1 = E_2$ and $E_3 = I_{n-k - k',  k - k'}$ so that $Z$ has the desired form: Let $\Sigma = \diag(\sigma_1,\dots,\sigma_{k'})$. When $\sigma_i \in (0,1)$, $\Sigma(I_{k'}-\Sigma)(E_1-E_2)=0$ forces the diagonal entries $(E_1)_{ii}=(E_2)_{ii}$.  When $\sigma_i=1$, the values of the diagonal entries $(E_1)_{ii}$ and $(E_2)_{ii}$ has no effect on the matrix $Z$ in \eqref{eq:Z1}, and we may replace $(E_2)_{ii}$ with $-(E_2)_{ii}$.  When $\sigma_i = 0$ and $(E_1)_{ii} (E_2)_{ii} = -1$,  the $2 \times 2$ submatrix of $Z$ formed by the $i$th and $(k'+i)$th rows and columns is either $\begin{bsmallmatrix}
1 & 0 \\
0 & -1
\end{bsmallmatrix}$ or $\begin{bsmallmatrix}
-1 & 0 \\
0 & 1
\end{bsmallmatrix}$.  Since $k + k' \le n$ and $\tr \big( \sqrt{I_{k'} - \Sigma^2} (E_1 - E_2) \bigr)  - \tr(E_3) = 2k - n$,  we may interchange $(E_2)_{ii}$ with a diagonal entry of $E_3$ to get a new $E_2$ with $(E_2)_{ii} = (E_1)_{ii}$. By repeating these steps, we arrive at $E_1 = E_2$.
\end{proof}

We will next develop some basic calculus for working with $\Phi(k,k',\mathbb{R}^n)$. As will be evident from the proof (but with Theorem~\ref{thm:char_C_M2} taking the place of Theorem~\ref{thm:char}), the following properties also hold with $\mathbb{C}^n$ in place of $\mathbb{R}^n$.
\begin{corollary}\label{cor:char}
Let $k, k', n \in \mathbb{N}$ with $k,k'  \le n$.  We have 
\begin{align*}
        \Phi(n-k,  k',\mathbb{R}^n) &= \Phi(k,  n-k',\mathbb{R}^n) =- \Phi(k,k',\mathbb{R}^n), \\
        \Phi(n-k, n-k',\mathbb{R}^n) &= \Phi(k,k',\mathbb{R}^n),\\
        \Phi(k',k,\mathbb{R}^n) &= \Phi(k,k',\mathbb{R}^n).
    \end{align*}
\end{corollary}
\begin{proof}
The first three equalities are consequence of the fact that in the involution model \eqref{eq:invo}, $\Gr(n-k,\mathbb{R}^n) = - \Gr(k,\mathbb{R}^n)$. It remains to establish the last equality.
For $(X,Y) \in \Gr(k,\mathbb{R}^n) \times \Gr(k',\mathbb{R}^n)$,  we have 
\[
 \Phi(k',k,\mathbb{R}^n) \ni Y X = (X Y)^\tp \in \Phi(k,k',\mathbb{R}^n)^\tp.
\]
Without loss of generality, let $k' \le k$.  If $k + k' \le n$,  then it follows from Theorem~\ref{thm:char}  that $\Phi(k,k',\mathbb{R}^n)^\tp = \Phi(k,k',\mathbb{R}^n)$,  from which we have 
\[
\Phi(k',k,\mathbb{R}^n) = \Phi(k,k',\mathbb{R}^n)^\tp = \Phi(k,k',\mathbb{R}^n).
\]
If $n < k + k'$,  then $(n-k) + (n-k') < n$ and $n-k \le n-k'$.  It follows from Theorem~\ref{thm:char} that $\Phi(n-k',n-k,\mathbb{R}^n)^\tp = \Phi(n-k',n-k,\mathbb{R}^n) $, and therefore
\begin{align*}
    \Phi(k,k',\mathbb{R}^n) &= \Phi(n-k,n-k',\mathbb{R}^n) = \Phi(n-k',n-k,\mathbb{R}^n)^\tp \\
    &= \Phi(n-k',n-k,\mathbb{R}^n) = \Phi(k',k,\mathbb{R}^n),
\end{align*}
as required.
\end{proof}

Let $k', k, n \in \mathbb{N}$ with $k \le k' \le n$ and $k+k' \le n$.  For any $k \times k$ diagonal matrices $\Sigma$ and $E$ with diagonal entries respectively in $[0,1]$ and $\{-1,1\}$,  we define the $\O(n)$-orbit
\[
\Orb_{k'} (\Sigma,E) \coloneqq \left\lbrace
Q \begin{bmatrix}
\sqrt{I_k - \Sigma^2} E  & \Sigma & 0  \\
- \Sigma  &  \sqrt{I_k - \Sigma^2} E & 0 \\
0 & 0 & I_{n-k' - k,  k' - k}
\end{bmatrix} Q^\tp \in \O(n):  Q\in  \O(n)
\right\rbrace.
\] 
The $3 \times 3$ block matrix in the middle is in $\O(n)$ by our choice of $\Sigma$ and $E$.  Note that there are $2^k$ possibilities for the $k \times k$ diagonal matrix $E$ with diagonal entries $\pm 1$ and two different $E$'s, even with the same signature, may give different orbits, as $
\tr(\sqrt{I_k-\Sigma^2}E)$ may be different for different $E$'s. We will find another representative for this orbit from which we may deduce its dimension later.
\begin{lemma}\label{lem:orbit}
Let $\Sigma \in \mathbb{R}^{k \times k}$ be the diagonal matrix
\[
\Sigma = \diag(\sigma_0 I_{m_0},   \dots,  \sigma_{s+1} I_{m_{s+1}}) 
\]  
where $1 = \sigma_0  >  \sigma_1 > \cdots > \sigma_s > \sigma_{s+1} = 0$ and $m_0,m_{s+1}\in \mathbb{N} \cup \{0\}$ and $m_1,\dots,m_s \in \mathbb{N}$ satisfying $ m_0+ \dots + m_{s+1} = k$.  Let $m \coloneqq  m_0 + \dots + m_s$.  Then for $k'\in \mathbb{N}$ with $k' \ge k$ and $k+k' \le n$,
\[
\Orb_{k'}(\Sigma,E) = \lbrace
Q F Q^\tp: Q\in \O(n)
\rbrace,
\]
for some
\begin{equation}\label{lem:orbit:eq1}
F \coloneqq
 \begin{bsmallmatrix}
0&  I_{m_0}    & 0 & 0 & \cdots  & 0 & 0 & 0   \\
- I_{m_0} & 0   & 0 & 0 & \cdots & 0 & 0 & 0   \\
0 & 0 & \sqrt{1-\sigma_1^2} I_{p_1,q_1} & \sigma_1 I_{m_1}    & \cdots & 0 & 0  & 0   \\
0 & 0 &  - \sigma_1 I_{m_1}  &  \sqrt{1-\sigma_1^2} I_{p_1,q_1}  & \cdots & 0 & 0 & 0   \\[-1ex]
 \vdots & \vdots & \vdots & \ddots & \ddots & \vdots & \vdots & \vdots \\[0.7ex]
0 & 0 & 0  & 0 & \cdots  &   \sqrt{1-\sigma_s^2} I_{p_s,q_s} & \sigma_s I_{m_s} & 0   \\
0 & 0 & 0  & 0 &  \cdots  &  -\sigma_s I_{m_s}   & \sqrt{1-\sigma_s^2} I_{p_s,q_s} & 0   \\
0 & 0 & 0  & 0 &  \cdots  & 0  &0  & I_{q,n-2m-q} 
\end{bsmallmatrix} \in \mathbb{R}^{n \times n}
\end{equation}
where $q \in \mathbb{N}$ with $n-k'-k \le q \le n-2m$ and $p_j, q_j \in \mathbb{N}$ satisfy $p_j + q_j = m_j$,  $j=1,\dots, s$.
\end{lemma}
\begin{proof}
Let $\Sigma_1 = \diag(\sigma_0 I_{m_0},   \dots,  \sigma_s I_{m_s})$  so that $\Sigma = \diag(\Sigma_1,  0_{m_{s+1}})$.  By conjugating
\[
F_2 =
\begin{bmatrix}
\sqrt{I_k - \Sigma^2} E  & \Sigma & 0  \\
- \Sigma  &  \sqrt{I_k - \Sigma^2} E & 0 \\
0 & 0 & I_{n-k' - k,  k' - k}
\end{bmatrix} \in \Orb_{k'}(\Sigma,E)
\]
with a permutation matrix $P$ (which has the effect of simultaneously permuting the rows and columns), we can bring it into the form
\[
F_1 = P F_2 P^\tp =
\begin{bmatrix}
 \sqrt{I_{m} - \Sigma_1^2} D_1 & \Sigma_1    & 0 & 0 & 0 \\
 -\Sigma_1  &  \sqrt{I_{m} - \Sigma_1^2} D_1  & 0 & 0 & 0  \\
 0  & 0 & D_2 & 0  & 0  \\
 0  & 0 & 0  & D_2  & 0  \\
 0  & 0 & 0  & 0  & I_{n-k' - k,  k' - k}  \\
\end{bmatrix}
\]
for some diagonal matrices $D_1\in \mathbb{R}^{m \times m}$ and $D_2 \in \mathbb{R}^{m_{s+1} \times m_{s+1} }$ with diagonal entries $\pm1$. Let $q$ be the number of $1$'s on the diagonal of $\diag(D_2,  D_2,  I_{n-k' - k,  k' - k})$, which can be brought into $I_{q,n-2m-q}$ by another permutation. 

The above procedure extracts the block $\sigma_{s+1} I_{m_{s+1}}$ out of $\Sigma$ and contributes to the bottom-right block $I_{q, n-2m-q}$ of $F$.  By repeating this procedure to extract the blocks $\sigma_1I_{m_1}, \dots, \sigma_sI_{m_s}$ out of $\Sigma$, the resulting matrix will have the required form in \eqref{lem:orbit:eq1}.
\end{proof} 

\begin{proposition}\label{prop:orb-dim}
Let $k, k', n,  q, m_0$ and $p_j,  q_j$,  $j = 1,\dots, s$,  be as in Lemma~\ref{lem:orbit}.  Then 
\[
\dim \Orb_{k'}(\Sigma,E) = \binom{n}{2} - \biggl[ m_0^2 +  \sum_{j=1}^s (p_j^2 + q_j^2) + \binom{q}{2} + \binom{n - q - 2m}{2}
\biggr].
\]
\end{proposition}
\begin{proof}
It suffices to determine the dimension of the Lie algebra $\mathfrak{g}_F$ of the isotropy group of $F$ in \eqref{lem:orbit:eq1}. Partition $F$ as $F = \diag(Q_0,\dots,  Q_{2s+2})$ where 
\[
Q_j =  \begin{cases}
\begin{bsmallmatrix}
0  &  I_{m_0}     \\
 -  I_{m_ 0}  & 0 \\
\end{bsmallmatrix} &\text{if } j = 0, \\[2ex]
\begin{bsmallmatrix}
\sqrt{\smash[b]{1-\sigma_i^2}} I_{p_i} & \sigma_i I_{p_i}     \\
 - \sigma_i I_{p_i}  &  \sqrt{\smash[b]{1-\sigma_i^2}} I_{p_i} \\
\end{bsmallmatrix} &\text{if }1 \le j \le 2s \text{~and~} j=2i-1, \\[3ex]
\begin{bsmallmatrix}
-\sqrt{\smash[b]{1-\sigma_i^2}} I_{q_i} & \sigma_i I_{q_i}  \\
 - \sigma_i I_{q_i}  & - \sqrt{\smash[b]{1-\sigma_i^2}} I_{q_i} \\
\end{bsmallmatrix} &\text{if }1 \le j \le 2s \text{~and~} j=2i, \\[2ex]
I_q &\text{if } j = 2s+1, \\[0.5ex]
-I_{n-2m-q} &\text{if } j = 2s+2.  \\
\end{cases}
\]
Let $X \in \mathsf{\Lambda}^2(\mathbb{R}^n)$ be such that $X F - F X =0$.  
We partition $X = (X_{ij})_{i,j=0}^{2s+2}$ accordingly so that
\begin{equation}\label{prop:orb-dim:eq2}
X_{ij} Q_j - Q_i X_{ij} = 0,\quad i,  j = 0,\dots, 2s+2.
\end{equation} 
If $i \ne j$,  then $\sigma(Q_i) \cap \sigma(Q_j) = \varnothing$. Thus the Sylvester equation \eqref{prop:orb-dim:eq2} gives $X_{ij} = 0$ \cite[Theorem 2.4.4.1]{HJ} whenever $ i \ne j$ and  $X$ is a block diagonal skew-symmetric matrix.  It remains to consider the diagonal blocks:
\begin{equation}\label{prop:orb-dim:eq3}
X_{jj} Q_j - Q_j X_{jj} = 0,\quad j = 0,\dots, 2s+2.
\end{equation} 
For $j = 2s +1$,  the solution space of \eqref{prop:orb-dim:eq3} has dimension $\binom{q}{2}$; for $j = 2s + 2$, it has dimension $\binom{n-2m-q}{2}$.  For $ j=0, \dots, 2s$,  \eqref{prop:orb-dim:eq3} reduces to
\begin{equation}\label{prop:orb-dim:eq4}
\begin{bmatrix}
b I_p & a I_p      \\
 - a I_p  &  b I_p  \\
\end{bmatrix} Y - Y
\begin{bmatrix}
b I_p & a I_p      \\
 - a I_p  &  b I_p  \\
\end{bmatrix} = 0
\end{equation} 
for some $a,b\in \mathbb{R}$ with $0 < a \le 1$ and $a^2 + b^2 =1$.  A direct calculation shows that the solution space of \eqref{prop:orb-dim:eq4} has dimension $p^2$.  Hence
\[
\dim \mathfrak{g}_F = m_0^2 + \sum_{j=1}^s (p_j^2 + q_j^2) + \binom{q}{2} + \binom{n - 2m -q}{2}.
\]
\end{proof}

The orbit $\Orb_{k'}(\Sigma,  E)$ is almost uniquely labeled by $\Sigma$ and $E$ in the following sense.
\begin{corollary}\label{cor:orb-dim-1}
$\Orb_{k'}(\Sigma,  E)  =  \Orb_{k'}(\Sigma',  E')$ 
if and only if there is a permutation matrix $P \in \mathbb{R}^{k \times k}$ such that $\Sigma' = P\Sigma P^\tp$ and $E' = P  E P^\tp$.
\end{corollary}
\begin{proof}
By Lemma~\ref{lem:orbit} and the proof of Proposition~\ref{prop:orb-dim},  an element in $\Orb_{k'}( \Sigma,  E) $ has eigenvalues 
\[
\ii,\quad -\ii, \quad 1, \quad -1,\quad \sqrt{\smash[b]{1- \sigma_j^2}} \pm \ii \sigma_j,\quad -\sqrt{\smash[b]{1- \sigma_j^2}} \pm \ii \sigma_j 
\]
with multiplicities $m_{0}$, $m_{0}$, $q$, $n-2m - q$, $p_j$, $q_j$ respectively, and where $j=1,\dots, s$.  Conversely,  any orthogonal matrix with these eigenvalues must be an element of $\Orb_{k'}( \Sigma,  E)$.  Therefore
\[
\Orb_{k'}(\Sigma,  E) \cap \Orb_{k'}(\Sigma',  E') \ne \varnothing
\]
if and only if $(\Sigma',  E')$ can be obtained from $(\Sigma,  E)$ by conjugating with the same permutation matrix.
\end{proof}

It follows from Corollary~\ref{cor:orb-dim-1} that if the diagonal matrix $\Sigma$ has distinct  diagonal entries, then it determines $\Orb_{k'}(\Sigma,  E)$ uniquely. When these distinct diagonal entries are all contained in $(0,1)$, the orbit $\Orb_{k'}(\Sigma,  E)$ turns out to be a flag manifold. Its dimension serves as an independent confirmation for our dimension formula in Proposition~\ref{prop:orb-dim} in this special case.
\begin{corollary}\label{cor:orb-dim-2}
Let $\Sigma, E$ be $k \times k$ diagonal matrices whose diagonal entries are in $(0,1)$ and $\{-1,1\}$  respectively. Let $k'\in \mathbb{N}$ be such that $k'\ge k$ and $k+k' \le n$. If the diagonal entries of $\Sigma$ are distinct, then
\[
\Orb_{k'}(\Sigma,  E)  \cong  \Flag(2,4,\dots,  2k,  n + k-k';\mathbb{R}^n). 
\]
In particular, 
\[
\dim \Orb_{k'} ( \Sigma,  E) = \binom{n}{2} - \biggl[ k + \binom{k'-k}{2} + \binom{n-k-k'}{2}
\biggr].
\]
\end{corollary}
\begin{proof}
For $1 > \sigma_1 > \cdots > \sigma_k > 0$, it follows from the proof of Proposition~\ref{prop:orb-dim} that the isotropy group of $F$ under the conjugation of $\O(n)$ is
\[
\underbrace{\O(2) \times \cdots \times \O(2)}_{k\text{ copies}} \times \O(n-k-k') \times \O(k'-k).
\] 
Therefore $\dim \Orb_{k'} ( \Sigma,  E) \cong \O(n)/ ( \O(2) \times \cdots \times \O(2) \times \O(n-k-k') \times \O(k'-k )$. Its dimension is then a consequence of the dimension of a flag manifold \cite{ZLK22}.
\end{proof}
We now arrive at the main result of this section.
\begin{theorem}[Special orthogonal group as product of two Grassmannians]\label{thm:prod-dim}
Let $k, k', n \in \mathbb{N}$ with $k' \le k < n$ and $k + k' \le n$.  Then
\begin{equation}\label{eq:dim1}
\dim \Phi(k, k', \mathbb{R}^n) = k(n-k) + k'(n-k') - k'
\end{equation}
and
\[
\Phi(k,k', \mathbb{R}^n) = \begin{cases}
\SO(n) &\text{if }k = k' \text{ and } n = 2k \text{ or } 2k + 1,  \\
\SO^\m (n) &\text{if }k = k'+1 \text{ and } n = 2k - 1 \text{ or } 2k.
\end{cases}
\]
\end{theorem}
\begin{proof}
By Theorem~\ref{thm:char},  we have 
\[
\Phi(k,  k', \mathbb{R}^n) = \bigcup_{(\Sigma, E) \in S} \Orb_k  (\Sigma, E) 
\]
where $S$ consists of pairs $(\Sigma, E)$ of $k' \times k'$ diagonal matrices whose diagonal entries are respectively in $[0,1]$ and $\{-1,1\}$. Since $k + k' \le n$,  by Corollary~\ref{cor:orb-dim-2},
\[
\dim \Orb_k  (\Sigma,  E) = \binom{n}{2} - \biggl[ k' + \binom{n-k- k'}{2} + \binom{k - k'}{2} \biggr].
\] 
Let $S'$ be the subset of $S$ consisting of $(\Sigma, E)$ with diagonal entries of $\Sigma$ are in $(0,1)$ and distinct. Then it is clear that 
\[
\bigcup_{(\Sigma, E) \in S’} \Orb_k  (\Sigma, E)
\]
contains an open subset of $\Phi(k,  k', \mathbb{R}^n) $. So by Corollary~\ref{cor:orb-dim-1},
\[
\dim \Phi(k,  k', \mathbb{R}^n) = k' + \binom{n}{2} - \biggl[ k' + \binom{n-(k + k')}{2} + \binom{k - k'}{2} \biggr],
\]
giving us \eqref{eq:dim1}. This additivity of dimension is an analogue of the fiber-dimension formula \cite[Corollary~11.13]{harris1992Algebraic} or quotient manifold theorem \cite[Theorem~21.10]{Lee03}.
Since $\O(n)$ has two connected components $\SO(n)$ and $\SO^\m (n)$,  each of dimension $\binom{n}{2}$, and $\Phi(k,  k', \mathbb{R}^n)$ is connected, we must have $\Phi(k,  k', \mathbb{R}^n) \subsetneq \SO(n)$ or $\Phi(k,  k', \mathbb{R}^n) \subsetneq \SO^\m(n)$ unless $k - k'$ and $n - k - k' \in \{0,1\}$.  If $k = k' $ and $n - 2k \in \{0,1\}$, then $I \in \Phi(k,k',\mathbb{R}^n)$ by Theorem~\ref{thm:char}.  This shows that $\Phi(k',k',\mathbb{R}^n) \subseteq \SO(n)$. By applying the real Schur decomposition \cite[Section~B.5]{Higham} to orthogonal matrices, Theorem~\ref{thm:char} implies that the equality must hold. If $k = k' + 1$ and $n - 2k + 1\in \{0,1\}$,  then $I_{n-1,1} \in \Phi(k,k',\mathbb{R}^n)$ by Theorem~\ref{thm:char} and so $\Phi(k,k',\mathbb{R}^n) = \SO^\m (n)$.  
\end{proof}
As a sanity check, for $k'=0$, we have $\dim \Phi(k, 0, \mathbb{R}^n) = k (n-k) = \dim \Gr(k,\mathbb{R}^n)$ as expected.

\section{Product of two complex Grassmannians}\label{sec:dim-C}

Here we discuss the complex analog of Section~\ref{sec:dim} in preparation for the next section. The main difference is that in this case, the product of two complex Grassmannians, $\Phi(k,k',\mathbb{C}^n)$ will not give us $\SU(n)$ regardless of the values of $k$ and $k'$. However the complex equivalent of several results in Section~\ref{sec:dim} still hold true with slight variation, such as the canonical form in Theorem~\ref{thm:char}, whose complex analog we prove below. Corollary~\ref{cor:char} also holds verbatim with $\mathbb{C}^n$ in place of  $\mathbb{R}^n$.
\begin{theorem}[Canonical form for product of two complex Grassmannians]\label{thm:char_C_M2}
Let $k, k', n \in \mathbb{N}$ with $k' \le k < n$ and $k + k' \le n$. Then $Z  \in \Phi(k,  k', \mathbb{C}^n)$ if and only if there exist $Q  \in \U(n)$ and a diagonal matrix $D = \diag( e^{\ii \theta_1},\dots, e^{\ii \theta_{k'}})$ with $\theta_j \in [0,\pi]$ such that 
\[ Z = 
Q \begin{bmatrix}
D  & 0 & 0  \\
0  &  \overline{D} & 0 \\
0 & 0 & I_{n-k - k',  k - k'}
\end{bmatrix} Q^\h .\]
\end{theorem}

\begin{proof}
Following the same derivation in Section~\ref{sec:dim}, we arrive at
\[
Z = Q_0 \begin{bmatrix}
        \sqrt{I_{k'} - \Sigma^2} E  & \Sigma & 0  \\
        - \Sigma  &  \sqrt{I_{k'} - \Sigma'^2} E & 0 \\
        0 & 0 & I_{n-k - k',  k - k'}
        \end{bmatrix} Q_0^\h.
\]
Unitarily diagonalizing the matrix in the middle brings $Z$ into the required form.
\end{proof}

To deduce the (real) dimension of $\Phi(k,k',\mathbb{C}^n)$, we count the number of real parameters required to describe $\Phi(k,k', \mathbb{C}^n)$  following the same approach as in Section~\ref{sec:dim}. Let $k, k', n \in \mathbb{N}$ with $k' \le k \le n$ and $k+k' \le n$. For any unitary diagonal matrix $D \in \mathbb{R}^{k'\times k'}$,  we define the $\U(n)$-orbit
\[
\Orb_k (D) \coloneqq \left\lbrace
Q \begin{bmatrix}
D  & 0 & 0  \\
0  &  \overline{D} & 0 \\
0 & 0 & I_{n-k - k',  k - k'}
\end{bmatrix} Q^\h  \in \U(n):  Q\in  \U(n)
\right\rbrace.
\]
In this section, $\Orb$ will always denote $\U(n)$-orbit (as opposed to $\O(n)$-orbit in the last section).
\begin{proposition}\label{prop:orb_C_M2-dim}
    Let $k, k', n \in \mathbb{N}$ with $k' \le k \le n$ and $k+k' \le n$. Let $D\in \mathbb{C}^{k'\times k'}$ be the diagonal matrix
\[
D = \diag(e^{\ii\theta_0} I_{m_0}, \dots, e^{\ii\theta_{s+1}} I_{m_{s+1}})
\]
with $0 = \theta_0  <  \theta_1 < \cdots < \theta_s < \theta_{s+1} = \pi$, and $m_0+\dots+m_{s+1}=k'$.  Let $m \coloneqq m_1+\dots+m_s $.  Then 
    \[
    \Orb_k (D) = \lbrace
    Q F Q^\h : Q\in \U(n)
    \rbrace,
    \]
    for some 
\[
    F \coloneqq \diag(e^{\ii\theta_1} I_{m_1},\cdots, e^{\ii\theta_s} I_{m_s}, e^{-\ii\theta_1} I_{m_1},\cdots, e^{-\ii\theta_s} I_{m_s}, I_{q,n-2m-q}) \in \mathbb{C}^{n\times n}
\]
where $q \coloneqq n-k-k'+2m_0 \in \mathbb{N} \cup \{0\}$. Furthermore
\[
\dim \Orb_k (D) = n^2 - \biggl[ 2 \sum_{j=1}^s m_j^2 + q^2 + (n-q-2m)^2 \biggr].
\]
\end{proposition}
\begin{proof}
The required $F$  follows from simply reordering the diagonal entries of $\diag(D, \overline{D}, I_{n-k-k',k-k'})$. This is the complex analog of Lemma~\ref{lem:orbit}. Also, as in Proposition~\ref{prop:orb-dim}, the required dimension follows from the dimension of the Lie algebra $\mathfrak{g}_F$ of the isotropy group of $F$:
\[
\mathfrak{g}_F = \{ X: XF = FX, \; X+X^\h =0 \}.
\]
Write $F$ as $F = \diag(F_1,\dots,  F_{2s+2})$ where 
\[
F_j =  \begin{cases}
e^{\ii\theta_j} I_{m_j} &\text{if }1 \le j \le s , \\[1ex]
e^{-\ii\theta_{j-s}} I_{m_{j-s}} &\text{if } s+1 \le j \le 2s, \\[1ex]
I_q &\text{if } j = 2s+1, \\[1ex]
-I_{n-2m-q} &\text{if } j = 2s+2,
\end{cases}
\]
and partition $X = (X_{ij})_{i,j=1}^{2s+2}$ accordingly. Note that $F_i$ and $F_j$ have distinct eigenvalues whenever $i\ne j$. By the same argument in the proof of Lemma~\ref{lem:orbit}, we have $X_{ij} = 0$ whenever $i\ne j$ and $X_{ii} + X_{ii}^\h =0$, $i,j =1,\dots, 2s+2$. Hence
\[
\dim \mathfrak{g}_F =  2\sum_{j=1}^s m_j^2 + q^2 + (n - 2m -q)^2.
\]
\end{proof}
We may now deduce the complex analog of \eqref{eq:dim1}.
\begin{theorem}\label{thm:prod_C_M2-dim}
Let $k, k', n \in \mathbb{N}$ with $k' \le k <n$ and $k+k'\le n$. Then
\[
\dim \Phi(k,k',\mathbb{C}^n) =2k(n-k) + 2k'(n-k') - k'.
\]
\end{theorem}
\begin{proof}
    By Theorem~\ref{thm:char_C_M2}, we have
\[
\Phi(k,k',\mathbb{C}^n) = \bigcup_{D \in S} \Orb_k  (D)
\]
where $S=\{ \diag(e^{\ii\theta_1},\dots, e^{\ii\theta_{k'}}) :  \theta_1,\dots,\theta_{k'} \in [0,\pi] \}$.   Following the line of argument in the proof of Theorem~\ref{thm:prod-dim}, we consider the subset $S' = \{\diag(e^{\ii\theta_1},\dots, e^{\ii\theta_{k'}}) :  \theta_1,\dots,\theta_{k'} \in (0,\pi)$ distinct$\}$.  The union of  $\Orb_k(D)$ over all $D \in S'$ is open and each $\Orb_k(D)$ has dimension $n^2 - 2k' - (n-k-k')^2 - (k-k')^2$ by Proposition~\ref{prop:orb_C_M2-dim}. Hence
\begin{align*}
        \dim \Phi(k,  k',\mathbb{C}^n) &= k' + [ n^2 - 2k' - (n-k-k')^2 - (k-k')^2 ]
         \\ 
        &= 2k(n-k) + 2k'(n-k') - k'.
\end{align*}
\end{proof}
As a sanity check, for $k'=0$, we have $\dim \Phi(k, 0, \mathbb{C}^n) = 2k (n-k) = \dim \Gr(k,\mathbb{C}^n)$ as expected. The reader may have noticed that the results in this section are considerably  neater and the proofs considerably simpler than their real counterparts in Section~\ref{sec:dim}; this is a consequence of the diagonalizability of matrices over $\mathbb{C}$.

\section{Product of complex Grassmannians}

The dimension in Theorem~\ref{thm:prod_C_M2-dim} shows that over $\mathbb{C}$, the product of two Grassmannians does not give the whole of $\SU(n)$. We will show in Proposition~\ref{prop:M3_fail} that neither does the product of three Grassmannians, but four would do it, as we will see in Theorem~\ref{thm:four}. We begin by generalizing the permutation invariance in Corollary~\ref{cor:char}.
\begin{corollary}\label{cor:permutation}
Let $k_1, \dots, k_d, n \in \mathbb{N}$ with $k_1,\dots, k_d < n$.
    \begin{enumerate}[label={\upshape(\roman*)}, ref=\roman*] %[\upshape (i)]
        \item\label{cor:permutation:item1}  For any permutation $\sigma \in \mathfrak{S}_d$, we have
            \begin{align*}
                \Phi(k_{\sigma(1)}, k_{\sigma(2)},\dots, k_{\sigma(d)},\mathbb{C}^n) &= \Phi(k_1,k_2,\dots,k_d,\mathbb{C}^n),  \\ 
                \Phi(n-k_1,k_2,\dots,k_d,\mathbb{C}^n) &= -\Phi(k_1,k_2,\dots,k_d,\mathbb{C}^n).
            \end{align*}
        \item\label{cor:permutation:item2} There exist $k_1',\dots,k_d' \in \mathbb{N}$ such that $n > k_1' \ge k_2' \ge \dots \ge k_d'$, $k_1'+k_2' \le n $, and
        \[
        \Phi(k_1',\dots,k_d',\mathbb{C}^n) = \Phi(k_1,\dots,k_d,\mathbb{C}^n).
        \] 
    \end{enumerate}
\end{corollary}
\begin{proof}
Let $X_1 \cdots X_d \in\Phi(k_1,\dots,k_d,\mathbb{C}^n)$. For any $j=1,\dots,d$, there exist $X_j' \in\Gr(k_j, \mathbb{C}^n)$, $X_{j+1}' \in \Gr(k_{j+1}, \mathbb{C}^n)$ such that $X_j X_{j+1} = X_{j+1}' X_j'$ by Corollary~\ref{cor:char}. So $\Phi(k_1,\dots, k_j,k_{j+1},\dots,k_d,\mathbb{C}^n) = \Phi(k_1,\dots,k_{j+1},k_j,\dots,k_d,\mathbb{C}^n)$, and \eqref{cor:permutation:item1} follows from repeating this argument. Suppose we have sorted in descending order $k_1\ge \dots \ge k_d$. If $k_1 + k_2 > n$, replace $(k_1,k_2)$ by $(n-k_1,n-k_2)$, and sort to descending order again. This may be repeated until we obtain \eqref{cor:permutation:item2}.
\end{proof}

It follows from Corollary~\ref{cor:permutation}\eqref{cor:permutation:item2} that if we define, for any $n, d \in \mathbb{N}$,
\begin{equation}\label{def:Kset}
K_{n,d} \coloneqq \lbrace (k_1,\dots,k_d)  \in \mathbb{N}^d :  \; k_1+k_2 \le n, \; k_1 \ge \dots \ge k_d \rbrace,
\end{equation}
then each $\Phi(k_1,\dots,k_d,\mathbb{C}^n)$ can be represented by at least one $d$-tuple $(k_1,\dots,k_d) \in K_{n,d}$. 

A \emph{reflection matrix} $R\in\U(n)$ is simply an element of $\Gr(n-1, \mathbb{C}^n)$, i.e., $R$ is unitarily similar to $\diag(-1, I_{n-1})$. Reflection matrices generate $\SU(n) \cup \SU^\m (n)$, a normal subgroup of $\U(n)$.  For a matrix $A \in \SU(n) \cup \SU^\m (n)$, its \emph{reflection length}  \cite{djokovic_products_1979} is given by
\begin{equation}\label{def:length}
\ell(A) \coloneqq \min \lbrace m \in \mathbb{N} : \; A = R_1 \cdots R_m,\; R_j \in \Gr(n-1, \mathbb{C}^n)\rbrace;
\end{equation}
and for a subset $S \subseteq \SU(n) \cup \SU^\m (n)$, it is given by
\begin{equation}\label{def:length2}
\ell(S) \coloneqq \max \lbrace \ell(A) : A \in S \rbrace.
\end{equation}
We reproduce a result in \cite{djokovic_products_1979} for easy reference.
\begin{theorem}[Djokovi\'c--Malzan]\label{thm:length}
Let the eigenvalues of $A \in \SU(n) \cup \SU^\m (n)$ be $e^{\ii \theta_1(A)},\dots, e^{\ii \theta_n(A)}$ with $0 \le \theta_1(A) \le \dots \le \theta_n(A) < 2\pi$.
Then
\[
\ell(A) = \frac{1}{\pi} \max \biggr(\sum_{j=1}^n \theta_j(A), \sum_{j=1}^n \theta_j(A^\h) \biggr).
\]
\end{theorem}
With this, we may deduce the following for the cases of interest to us.
\begin{corollary}[Reflection length of Grassmannian]\label{cor:length}
Let $k,k_1,\dots,k_d, n \in \mathbb{N}$ with $k, k_1,\dots,k_d < n$, $k_1+k_2 \le n$, and $k_1 \ge \dots \ge k_d$. Then $\ell(A)=n-k$ for any $A\in \Gr(k, \mathbb{C}^n)$,
\[
\ell(\Gr(k, \mathbb{C}^n)) = n-k, \qquad \ell(\SU(n)) = 2n-2, \qquad \ell(\SU^\m (n)) = 2n-1,
\]
and
\[
\ell(\Phi(k_1,\dots,k_d,\mathbb{C}^n)) \le \begin{cases}
                k_1 + k_2 + \dots + k_d &\text{if } d \text{ is even}, \\
                n-k_1 + k_2 + \dots + k_d  &\text{if } d \text{ is odd}.
            \end{cases}
\]
\end{corollary}
\begin{proof}
The equalities all follow from Theorem~\ref{thm:length} immediately. The last inequality follows from Corollary~\ref{cor:permutation}: Since $\Phi(k_1,k_2,\mathbb{C}^n) = \Phi(n-k_1,n-k_2,\mathbb{C}^n)$ and $\ell(\Gr(k, \mathbb{C}^n)) = n-k$, we have $\ell(\Phi(k_1,k_2,\mathbb{C}^n))\le k_1+k_2$. Applying this to pairs of neighboring indices yields the result.
\end{proof}

We next obtain $\SU(n) = \Phi(k_1,k_2,k_3,k_4,\mathbb{C}^n)$ as a special case of a more general statement: Any $Z \in \SU(n) \cup \SU^\m (n)$ can be expressed as $Z=X_1X_2X_3X_4$ with $X_i \in \Gr(k_i, \mathbb{C}^n)$ for certain tuples $(k_1,k_2,k_3,k_4)$. We stress that this is more than a refinement of the result for special linear group in \cite{gustafson_products_1976} and cannot be deduced from it. For one, the involution matrices constructed in \cite{gustafson_products_1976} are neither symmetric nor Hermitian and do not have prescribed trace. We see no conceivable way to impose these conditions in \cite{gustafson_products_1976} without derailing their arguments.

\begin{lemma} \label{lem:suff_k2k4}
    For $2 \le n \in \mathbb{N}$, let $K_{n,d}$ be as defined in \eqref{def:Kset}. If $(k_1,k_2,k_3,k_4) \in K_{n,4}$ and $k_2+k_4 \ge n-1$, then 
    \[ \Phi(k_1,k_2,k_3,k_4,\mathbb{C}^n) = \begin{cases}
            \SU(n) &\text{if } k_1 + k_2 + k_3 + k_4 \text{~is~even}, \\
            \SU^\m (n) &\text{if }
            k_1 + k_2 + k_3 + k_4 \text{~is~odd}.
        \end{cases}
    \]
\end{lemma}
\begin{proof}
Suppose $n=2k$. Let  $C\in \SU(n)$. Without loss of generality,  we may assume that $C=\diag(e^{\ii \gamma_1} ,\dots,e^{\ii \gamma_n})$ with $\sum_{i=1}^n \gamma_i = 0$. Since $(k_1,k_2,k_3,k_4)\in K_{n,4}$ and $k_2+k_4 \ge n-1$, we must have $k_1=k_2=k$ and $(k_3,k_4) \in \{(k,k), (k,k-1),(k-1,k-1)\}$. Let
\[
\alpha_j \coloneqq \sum_{i=1}^{2j-1}\gamma_i,\quad j=1,\dots, k; \qquad \beta_j \coloneqq \sum_{i=1}^{2j} \gamma_i,  \quad  j=1,\dots, k-1.
\]
From these, we define two matrices:
\begin{align*}
A &\coloneqq \diag(e^{\ii \alpha_1},e^{-\ii \alpha_1},\dots,e^{\ii \alpha_{k}},e^{-\ii \alpha_k}),  \\
B &\coloneqq \diag(1, e^{\ii \beta_1},e^{-\ii \beta_1},\dots,e^{\ii \beta_{k-1}},e^{-\ii \beta_{k-1}},(-1)^{k_3-k_4}).
\end{align*}
By Theorem \ref{thm:char_C_M2},  we have $A \in \Phi(k_1,k_2,\mathbb{C}^{2k})$ and $B \in \Phi(k_3,k_4,\mathbb{C}^{2k})$. It is straightforward to verify that $C = AB$.  

Suppose $n=2k+1$. Let  $C \in \SU^\m (n)$. Without loss of generality,  we may assume that $C=\diag(e^{\ii \gamma_1} ,\dots,e^{\ii \gamma_n})$ with $\sum_{i=1}^n \gamma_i = \pi$.  The condition that $(k_1,k_2,k_3,k_4) \in K_{n,4}$ implies that $k_1\in \{k+1,k\}$ and $k_2=k_3=k_4=k$. Let
\[
\alpha_j \coloneqq \sum_{i=1}^{2j-1}\gamma_i,\quad j =1,\dots, k; \qquad
\beta_j \coloneqq \sum_{i=1}^{2j} \gamma_i,  \quad  j=1,\dots, k. 
\]
From these, we define two matrices:
\begin{align*}
    A &\coloneqq \diag(e^{\ii \alpha_1},e^{-\ii \alpha_1},\dots,e^{\ii \alpha_{k}},e^{-\ii \alpha_k}, (-1)^{k_1-k_2}), \\
    B &\coloneqq \diag(1, e^{\ii \beta_1},e^{-\ii \beta_1},\dots,e^{\ii \beta_{k}},e^{-\ii \beta_{k}}).
\end{align*}
By Theorem \ref{thm:char_C_M2}, we have $A \in \Phi(k_1,k_2,\mathbb{C}^{2k})$ and $B \in \Phi(k_3,k_4,\mathbb{C}^{2k})$.   It is straightforward to verify that $C = AB$.
\end{proof}
 
We next consider the cases not covered by Lemma~\ref{lem:suff_k2k4}, i.e., where $k_2 + k_4 \le n -2$. We begin by addressing a case that requires separate treatment, which requires the following two lemmas.
\begin{lemma} \label{lem:M3111_eqs}
For any $(\theta_1, \theta_2, \theta_3, \theta_4) \in \mathbb{R}^4$ with $\theta_1+\theta_2+\theta_3+\theta_4=0$, there exists a permutation $\tau$ on $\{1,2,3,4\}$ such that
\begin{align*}
x(\cos \theta_{\tau(1)} + \cos \theta_{\tau(2)}) - 2y \cos \theta_{\tau(3)} &= \cos \theta_{\tau(2)} - \cos \theta_{\tau(3)}, \\
x(\sin \theta_{\tau(1)} + \sin \theta_{\tau(2)})  + 2y \sin \theta_{\tau(3)} &= \sin \theta_{\tau(2)} + \sin \theta_{\tau(3)},
\end{align*}
has a solution $(x_*, y_*) \in [0,1] \times [0,1]$.
\end{lemma}
\begin{proof} 
We rewrite the system in a more compact form:
    \begin{equation} \label{lem:M3111_eqs:prf:eq1}
        e^{\ii\theta_{\tau(1)}}x - e^{\ii \theta_{\tau(2)}}(1-x)=e^{-\ii \theta_{\tau(3)}}(2y-1). 
\end{equation} 
We regard $e^{\ii\theta_1},e^{\ii\theta_2},e^{\ii\theta_3}, e^{\ii\theta_4}$ as points on the unit circle. We claim that there exists some permutation $\tau$ such that $e^{\ii \theta_{\tau(1)}}, e^{\ii \theta_{\tau(2)}},  e^{-\ii\theta_{\tau(3)}}$ lie in the same semicircle, and, traversing counterclockwise, $e^{-\ii\theta_{\tau(3)}}$ is either the first or last among the three points.  In this case,  \eqref{lem:M3111_eqs:prf:eq1} has a solution in $[0,1] \times [0,1]$: First, observe that there exists some $x_* \in [0,1]$ such that $\arg \left( e^{\ii\theta_{\tau(1)}}x_* - e^{\ii \theta_{\tau(2)}}(1- x_*) \right) = -\theta_{\tau(3)}+k\pi$, where $k \in \{0,1\}$ depends on the position of $e^{-\ii\theta_{\tau(3)}}$.  Next, we may find $y_* \in [0,1]$ such that $|e^{\ii\theta_{\tau(1)}}x_* - e^{\ii \theta_{\tau(2)}}(1- x_*)| = (-1)^k(2y_* - 1)$. So $(x_*,y_*)$ is a solution of \eqref{lem:M3111_eqs:prf:eq1}.

It remains to show that a $\tau$ exists with the desired property.  Consider the complementary semicircles 
\[
S_\p \coloneqq \{e^{\ii t}: -\theta_3 \le t < -\theta_3+\pi \},\quad S_\m \coloneqq \{e^{\ii t}: -\theta_3-\pi \le t < -\theta_3 \}. 
\]
By the Pigeonhole Principle,  at least two of $e^{\ii \theta_1},e^{\ii \theta_2},e^{\ii \theta_4}$ must lie in the same semicircle.  Thus,  there is a permutation $\tau$ such that $\tau(3) = 3$ and $e^{\ii \theta_{\tau(1)}},  e^{\ii \theta_{\tau(2)}}$ lie in the same semicircle, as required.
\end{proof}
\begin{lemma} \label{lem:M3111}
For $k \ge 2$, $\SU(2k) = \Phi(k,k,k,k-2,\mathbb{C}^{2k})$.
\end{lemma} 
\begin{proof}
We show that every $C \in \SU(2k)$ lies in $\Phi(k,k,k,k-2,\mathbb{C}^{2k})$. Without loss of generality,  we assume $C = \diag(e^{\ii \gamma_1},\dots,e^{\ii \gamma_{2k}})$ with $\gamma_1 + \dots + \gamma_{2k} = 0$. Denote
\[
\alpha_j \coloneqq \begin{cases}
    \sum_{i=2j+1}^{2k-3} \gamma_i &\text{for~} j=1,\dots,k-2, \\
    \sum_{i=1}^{2k-3} \gamma_i &\text{for~} j=k-1, \\
    \gamma_{2k-2} + \pi &\text{for~} j=k,
\end{cases}\quad  \beta_j \coloneqq \sum_{i=2j}^{2k-3} \gamma_i \text{~~for~} j=1,\dots,k-2.
\]
We define
\[
        V(x,y,\theta) = \begin{bmatrix}
        x & y \\
        -e^{\ii \theta} \overline{y} & e^{\ii \theta} \overline{x}
        \end{bmatrix}, \quad |x|^2 + |y|^2=1, \quad \theta \in [0,2\pi).
\]
Note that any element in $\U(2)$ can be parameterized as $V(x,y,\theta)$ for some parameters $x,y,\theta$.

Next,  we consider the matrix equation $C = AB$ where $A = \diag(A_1, A_2)$,  $B = \diag(B_1, B_2)$, and     
        \begin{align*}
        A_1 &\coloneqq \diag(e^{\ii \alpha_{k-1}}, e^{-\ii \alpha_1},e^{\ii \alpha_1}, e^{-\ii \alpha_2},e^{\ii \alpha_2},\dots,e^{-\ii \alpha_{k-3}},e^{\ii \alpha_{k-3}}, e^{-\ii \alpha_{k-2}}), \\
        A_2 &\coloneqq \begin{bmatrix}
        \diag(e^{\ii \alpha_{k-2}},e^{\ii \alpha_k}) & 0 \\
        0 & V_A 
        \diag(e^{-\ii \alpha_{k-1}}, e^{-\ii \alpha_k}) V_A^\h
        \end{bmatrix}, \\
     V_A &\coloneqq V(x_A, y_A, \theta_A), \\
        B_1 &\coloneqq \diag(e^{-\ii \beta_1},e^{\ii \beta_1},e^{-\ii \beta_2},e^{\ii \beta_2},\dots,e^{-\ii \beta_{k-2}}, e^{\ii \beta_{k-2}}), \\
        B_2 &\coloneqq \begin{bmatrix}
        I_{1,-1} & 0 \\
        0 & V_B I_{1,-1} V_B^\h 
        \end{bmatrix},\\
        V_B &\coloneqq V(x_B, y_B, \theta_B).
    \end{align*}
Here $x_A,  y_A,  \theta_A,  x_B,  y_B,  \theta_B$ are unknown parameters to be determined.  By Theorem~\ref{thm:char_C_M2},  we have $A \in  \Phi(k,k,\mathbb{C}^{2k})$ and $B \in \Phi(k,k-2,\mathbb{C}^{2k})$. Unravelling the definition,  the equation $C=AB$ becomes
\begin{equation}\label{lem:M3111:prf:eqs_C=AB}
\begin{aligned}
    (\cos \phi_1 + \cos \phi_2)|x_A|^2 - 2\cos \phi_3 |x_B|^2 &= \cos \phi_2 - \cos \phi_3,  \\
    (\sin \phi_1 + \sin \phi_2) |x_A|^2 + 2 \sin \phi_3 |x_B|^2 &= \sin \phi_2 + \sin \phi_3,   \\
    \arg\Bigl( \frac{e^{\ii\phi_3} x_B y_B}{(e^{-\ii\phi_1}+e^{-\ii\phi_2})x_A y_A} \Bigr) &= \theta_B - \theta_A,
\end{aligned}
\end{equation}
where $(\phi_1, \phi_2, \phi_3, \phi_4) = \Bigl(\sum_{i=1}^{2k-3}\gamma_i, \gamma_{2k-2},\gamma_{2k-1},\gamma_{2k}\Bigr)$.  By Lemma~\ref{lem:M3111_eqs},  we may permute the diagonal entries of $C$ so that the first two equations in \eqref{lem:M3111:prf:eqs_C=AB} has a solution $(|x_A|, |x_B|)$. Observe that the third equation  in \eqref{lem:M3111:prf:eqs_C=AB} does not place any restriction on the values of $|x_A|$ and $|x_B|$. So we are free to pick any $x_A,  y_A,  \theta_A,  x_B,  y_B,  \theta_B$ that satisfy the third equation in \eqref{lem:M3111:prf:eqs_C=AB}. These yield a desired decomposition of $C$. 
\end{proof}

We remind the reader that by its definition \eqref{eq:invo}, the product of $d$ copies of complex Grassmannians must necessarily be a subset of either  $\SU(n)$ or $\SU^\m(n)$ but not both. Our main result of this section may also be viewed as a complete characterization of when it is a proper subset when $d=4$.  
\begin{theorem}[Special unitary group as product of four Grassmannians]\label{thm:four}
For $k \in \mathbb{N}$,  we have 
\begin{align*}
\SU(2k) &= \Phi(k,k,k,k,\mathbb{C}^{2k}) = \Phi(k,k,k-1,k-1,\mathbb{C}^{2k}) = \Phi(k,k,k,k-2,\mathbb{C}^{2k}),  \\
\SU(2k+1) &= \Phi(k,k,k,k,\mathbb{C}^{2k+1}), \\
\SU^{\m}(2k) &= \Phi(k,k,k,k-1,\mathbb{C}^{2k}),  \\
\SU^{\m}(2k+1) &= \Phi(k+1,k,k,k,\mathbb{C}^{2k+1}). 
\end{align*} 
Here we assume $k \ge 2$ in $\Phi(k,k,k-1,k-1,\mathbb{C}^{2k})$ and $\Phi(k,k,k,k-1,\mathbb{C}^{2k})$,  and $k \ge 3$ in $\Phi(k,k,k,k-2,\mathbb{C}^{2k})$.  Moreover,  for any $(k_1,k_2,k_3,k_4) \in K_{n,4}$ not listed above,  $\Phi(k_1,k_2,k_3,k_4,\mathbb{C}^n)$ is a proper subset of $\SU(n)$ or $\SU^\m (n)$. 
\end{theorem}
\begin{proof}
If $(k_1, k_2, k_3, k_4) \in K_{n,4}$ is such that $\SU^\m (n)=\Phi(k_1,k_2,k_3,k_4,\mathbb{C}^{n})$, then we necessarily have
 \begin{align*}
                    \dim(\Phi(k_1,k_2,\mathbb{C}^{n})) + \dim(\Phi(k_3,k_4,\mathbb{C}^{n})) &\ge \dim(\SU^\m (n))=n^2-1, \\
                    k_1 + k_2 + k_3 + k_4 &\ge \ell(\SU^\m (n)) =2n-1,\\
                    k_1+k_2+k_3+k_4  & \equiv 1 \pmod 2,
 \end{align*}
by Corollary~\ref{cor:length} and Theorem~\ref{thm:prod_C_M2-dim}. For $n=2k$, the condition $(k_1,k_2,k_3,k_4) \in K_{n,4}$ implies that
\[
4k - 1= 2n-1 \le k_1+k_2+k_3+k_4 \le k_1+3k_2\le k_1 + 3(n-k_1) = 6k - 2k_1.
\]
This shows that $k_1 \le k$. Since we also have that $4k_1 \ge k_1+k_2+k_3+k_4 \ge 2n-1 = 4k - 1$,  we obtain $k_1=k$, which in turn forces $(k_1,  k_2,  k_3,  k_4) = (k,k,k,k-1)$. For $n=2k+1$,  a similar calculation shows that $(k_1,  k_2,  k_3,  k_4) = (k+1,k,k,k)$. By Lemma~\ref{lem:suff_k2k4},  we have $\SU^\m (2k)=\Phi(k,k,k,k-1,\mathbb{C}^{2k+1})$ and $\SU^\m (2k+1)=\Phi(k+1,k,k,k,\mathbb{C}^{2k+1})$. 

If $(k_1, k_2, k_3, k_4) \in K_{n,4}$ is such that $\SU(n)=\Phi(k_1,k_2,k_3,k_4,\mathbb{C}^n)$, then we necessarily have
\begin{align*}
    \dim(\Phi(k_1,k_2,\mathbb{C}^n)) + \dim(\Phi(k_3,k_4,\mathbb{C}^{n})) &\ge \dim(\SU(n))=n^2-1, \\
    k_1 + k_2 + k_3 + k_4 &\ge \ell(\SU(n)) =2n-2, \\
    k_1+k_2+k_3+k_4 &\equiv 0 \pmod 2.
\end{align*}
For $n=2k+1$, these inequalities have two common solutions: $(k,k,k,k)$, which satisfies Lemma~\ref{lem:suff_k2k4}, and $(k+1,k,k,k-1)$, which does not. To rule out the second case, note that $A \mapsto -A$ defines a bijection between $\SU(2k+1)$ and $\SU^\m (2k+1)$, and $-\Phi(k+1,k,k,k-1,\mathbb{C}^{2k+1})=\Phi(k,k,k,k-1,\mathbb{C}^{2k+1}) \subsetneq \SU^\m (2k+1)$. Hence $\Phi(k+1,k,k,k-1,\mathbb{C}^{2k+1}) \subsetneq \SU(2k+1)$.
  For $n=2k$, these inequalities have four common solutions: $(k,k,k,k)$, $(k,k,k-1,k-1)$, $(k,k,k,k-2)$, and $(k+1,k-1,k-1,k-1)$. Lemma~\ref{lem:suff_k2k4} shows that indeed $\SU(2k) = \Phi(k,k,k,k,\mathbb{C}^{2k}) = \Phi(k,k,k-1,k-1,\mathbb{C}^{2k})$. We have already established $\SU(2k) = \Phi(k,k,k,k-2,\mathbb{C}^{2k})$ in Lemma~\ref{lem:M3111}. To rule out the last remaining case, note that $A \mapsto -A$ defines a bijection on $\SU(2k)$ itself, and $-\Phi(k+1,k-1,k-1,k-1, \mathbb{C}^{2k})=\Phi(k-1,k-1,k-1,k-1,\mathbb{C}^{2k}) \subsetneq \SU(2k)$. Hence $\Phi(k+1,k-1,k-1,k-1,\mathbb{C}^{2k}) \subsetneq \SU(2k)$. 
\end{proof}

The decomposition into four Grassmannians is the best possible. We adapt \cite[Theorem~2]{halmos_products_1958} to show that for infinitely many $n \in \mathbb{N}$, $\SU(n)$ cannot be expressed as a product of three Grassmannians.
\begin{proposition}\label{prop:M3_fail}
For any $n=3m$, $m\in \mathbb{N}$, there exists $V\in \SU(n)$ such that $V \notin \Phi(k_1,k_2,k_3,\mathbb{C}^n)$ for any $(k_1, k_2, k_3) \in K_{n,3}$.
\end{proposition}
\begin{proof}
Suppose $V=e^{ \frac{2\pi }{3} \ii}I_n \in \SU(n)$ can be written as $V=X_1X_2X_3$ with $X_i\in \Gr(k_i, \mathbb{C}^n)$, $i =1,2,3$. Bearing in mind that $VX_i=X_iV$ for each $i$,  as $V$ is a multiple of the identity, we have
    \begin{align*}
        V^4 &= V(X_1 V) X_2 (V X_3) = V(X_2 X_3) X_2 (X_1 X_2) \\
        &= X_2 (V X_3) X_2 (X_1 X_2)= X_2 X_1 X_2 X_2 X_1 X_2 = I,
    \end{align*}
a contradiction.
\end{proof}

\section{Product of symplectic Grassmannians}

The symplectic Grassmannian represented in the form of symplectic involutions
\begin{equation}\label{eq:GrSp}
\Gr_{\Sp}(2k, \mathbb{F}^{2n}) \coloneqq  \{ X \in \Sp(2n,\mathbb{F}): X^2 = I_{2n},\; \tr(X) = 4k - 2n \}
\end{equation}
is not well known. In fact this article likely records its first appearance. We will establish some basic results here, largely to demonstrate that our definition in  \eqref{eq:GrSp} is isomorphic to the existing definition in \cite{LL11}. These results will hold true for $\mathbb{F} = \mathbb{R}$ and $\mathbb{C}$ alike.
\begin{lemma}[Symplectic Grassmannian as adjoint orbit]\label{lem:SPidentification}
Let $k, n \in \mathbb{N}$ with $k \le n$. Then
\[
\Gr_{\Sp}(2k, \mathbb{F}^{2n}) = \{ Q \diag(I_{k,n-k},  I_{k,n-k}) Q^{-1}: Q\in \Sp(2n,\mathbb{F}) \}.
\]
\end{lemma}
\begin{proof}
Since each $X \in \Gr_{\Sp}(2k,\mathbb{F}^{2n})$ is an involution,  it is diagonalizable by some $Q \in \GL_n(\mathbb{F})$, and the eigenvalues of $X$ are $1$ and $-1$,  with multiplicity $2k$ and $2n-2k$ respectively, i.e.,  $X = Q \diag(I_{k,n-k},  I_{k,n-k}) Q^{-1}$. Since $X \in \Sp(2n,\mathbb{F})$, we may choose its eigenbasis to be symplectic \cite[Exercise~2.2.7]{mcduff_salamon_2017}, i.e., $Q$ may be chosen to be in $\Sp(2n,\mathbb{F})$.
\end{proof}
We will show that $\Gr_{\Sp}(2k, \mathbb{F}^{2n})$ above is isomorphic to the more common \cite{LL11} homogeneous space description $\Sp(2n, \mathbb{F})/\bigl( \Sp(2k, \mathbb{F}) \times \Sp(2n-2k, \mathbb{F}) \bigr)$ in \eqref{eq:homo}. As will be evident from the proof below, the isomorphism may be taken in any appropriate sense: diffeomorphism of smooth manifolds, biholomorphic isomorphism of complex manifolds, biregular isomorphism of algebraic varieties, etc.
\begin{proposition}[Symplectic Grassmannian as homogeneous space]\label{prop:SPidentification}
Let $k, n \in \mathbb{N}$ with $k \le n$. Then
\[
\Gr_{\Sp}(2k, \mathbb{F}^{2n}) \cong \Sp(2n, \mathbb{F})/\bigl( \Sp(2k, \mathbb{F}) \times \Sp(2n-2k, \mathbb{F}) \bigr).
\]
\end{proposition}
\begin{proof}
Let $D \coloneqq \diag(I_{k,n-k},  I_{k,n-k})$.  The map
\[
\varphi: \Sp(2n,\mathbb{F}) \to \Gr_{\Sp}(2k,\mathbb{F}^{2n}),\quad \varphi(Q) = Q D Q^{-1},
\]
is surjective by Lemma~\ref{lem:SPidentification}, i.e., the adjoint action of $\Sp(2n,\mathbb{F})$ on $\Gr_{\Sp}(2k, \mathbb{F}^{2n})$ is transitive.  The required isomorphism would follow immediately if we can show that $\Stab_D(\Sp(2n,\mathbb{F})) \cong \Sp(2k,\mathbb{F}) \times \Sp(2n-2k,\mathbb{F})$.  We partition $Q = \begin{bsmallmatrix}
Q_1 & Q_2 \\
Q_3 & Q_4
\end{bsmallmatrix} \in \Stab_D(\Sp(2n,\mathbb{F}))$ into $n\times n$ blocks $Q_1,Q_2,  Q_3$ and $Q_4$.  Since $QD = DQ$,  we have $Q_i I_{k,n-k} = I_{k,n-k} Q_i$, $i=1,2,3, 4$.  We further partition each $Q_i$ into $Q_i = \begin{bsmallmatrix} X_i & Y_i \\ Z_i & W_i \end{bsmallmatrix}$ with $X_i \in \mathbb{F}^{k \times k}$, $Y_i \in \mathbb{F}^{k \times (n-k)}$,  $Z_i \in \mathbb{F}^{(n-k) \times k}$, $W_i \in \mathbb{F}^{(n-k) \times (n-k)}$, and from $Q_i I_{k,n-k} = I_{k,n-k} Q_i$ we immediately see that $Y_i = 0$ and $Z_i = 0$.  Since $Q^\tp J_{2n}Q = J_{2n}$, we have $Q_1^\tp Q_3 - Q_3^\tp Q_1 = Q_2^\tp Q_4 - Q_4^\tp Q_2 = 0$ and $Q_1^\tp Q_4 - Q_3^\tp Q_2 = I$.  Consequently, we obtain
\begin{align*}
X_1^\tp X_3 &= X_3^\tp X_1, & X_2^\tp X_4 &= X_4^\tp X_2, & X_1^\tp X_4 - X_3^\tp X_2 &= I_k; \\
W_1^\tp W_3 &= W_3^\tp W_1, \quad & W_2^\tp W_4 &= W_4^\tp W_2, \quad & W_1^\tp W_4 - W_3^\tp W_2 &= I_{n-k}.
\end{align*}
The equations involving $X_i$'s are equivalent to $\begin{bsmallmatrix} 
X_1 & X_2 \\ 
X_3 & X_4 
\end{bsmallmatrix} \in \Sp(2k,\mathbb{F})$ and those involving $W_i$'s are equivalent to $\begin{bsmallmatrix} 
W_1 & W_2 \\ 
W_3 & W_4 
\end{bsmallmatrix} \in \Sp(2n-2k,\mathbb{F})$.  Hence $\Stab_D(\Sp(2n,\mathbb{F})) \cong \Sp(2k,\mathbb{F}) \times \Sp(2n-2k,\mathbb{F})$.
\end{proof}

Finally, we will show that set theoretically, $\Gr_{\Sp}(2k, \mathbb{F}^{2n})$ indeed agrees with $\mathbb{G}_{\Sp}(2k, \mathbb{F}^{2n})$ in \eqref{eq:abstract}, i.e., matrices in $\Gr_{\Sp}(2k, \mathbb{F}^{2n})$ are in bijective correspondence with symplectic subspaces in $\mathbb{G}_{\Sp}(2k, \mathbb{F}^{2n})$.
\begin{corollary} [Symplectic Grassmannian as matrix manifold]\label{cor:SPidentification1}
Let $k, n \in \mathbb{N}$ with $k \le n$. Then
\begin{equation}\label{eq:set}
\Gr_{\Sp}(2k, \mathbb{F}^{2n}) \cong \mathbb{G}_{\Sp}(2k, \mathbb{F}^{2n}).
\end{equation}
\end{corollary}
\begin{proof}
Let $X\in \Sp(2n,\mathbb{F})$. Then $X = Q \diag(I_{k,n-k},  I_{k,n-k}) Q^{-1}$ as in  Lemma~\ref{lem:SPidentification}. We write $q_1,\dots,  q_{2n} \in \mathbb{F}^{2n}$ for  the column vectors of $Q$.  Define $\mathbb{V}_X$ to be the $2k$-dimensional subspace spanned by $q_1,\dots, q_k, q_{n+1},\dots,  q_{n+k}$. $\mathbb{V}_X$ is well-defined because,  as is evident from the proof of Proposition~\ref{prop:SPidentification},  if 
\[
Q \diag(I_{k,n-k},  I_{k,n-k}) Q^{-1} = Q' \diag(I_{k,n-k},  I_{k,n-k}) {Q'}^{-1},
\]
then $Q=Q' R$ for some matrix
\begin{equation}\label{cor:SPidentification1:eq1}
R =\begin{bmatrix}
    X_1 & 0 & X_2 & 0 \\
    0 & W_1 & 0 & W_2 \\
    X_3 & 0 & X_4 & 0 \\
    0 & W_3 & 0 & W_4
\end{bmatrix}
\end{equation} 
where $\begin{bsmallmatrix}
X_1 & X_2 \\
X_3 & X_4
\end{bsmallmatrix} \in  \Sp(2k,\mathbb{F}) $ and $\begin{bsmallmatrix}
W_1 & W_2 \\
W_3 & W_4
\end{bsmallmatrix} \in  \Sp(2n - 2k,\mathbb{F})$.
So
\[
\spn \{q_1,\dots,  q_k,  q_{n+1},\dots,  q_{n+k}\} = \spn \{q'_1,\dots,  q'_k,  q'_{n+1},\dots,  q'_{n+k}\}.
\]
We will call any such $Q$ an eigenbasis matrix of  $\mathbb{V}_X$. 

Next we prove that the map given by $\varphi(X) = \mathbb{V}_X$ gives the bijection in \eqref{eq:set}.  Note that $\varphi$ is injective: If $\mathbb{V}_{X_1} = \mathbb{V}_{X_2}$,  then their eigenbasis matrices $Q_1$ and $Q_2$ must satisfy $Q_2 = Q_1 R$ for some $R$ of the form \eqref{cor:SPidentification1:eq1}.  This implies that $X_2 = Q_2 \diag(I_{k,n-k},  I_{k,n-k}) Q_2^{-1}  = Q_1 \diag(I_{k,n-k},  I_{k,n-k}) Q_1^{-1} = X_1 $. The map $\varphi$ is surjective: Every $2k$-dimensional symplectic subspace $\mathbb{V}$ of $\mathbb{F}^{2n}$ has a unique orthogonal complement $\mathbb{V}^{\perp}$ with respect to the symplectic form given by $J_{2n}$.  Let $\{q_1,\dots,  q_k,  q_{n+1},  \dots,  q_{n+k}\}$ be a symplectic basis of $\mathbb{V}$ and $\{q_{k+1},\dots,  q_{n},  q_{n+k+1},  \dots,  q_{2n}\}$ a symplectic basis of $\mathbb{V}^{\perp}$. Let $Q$ be the matrix whose columns are $q_1,\dots,  q_{2n}$.  Then $Q \in \Sp(2n,\mathbb{F})$ and $\varphi(Q \diag(I_{k,n-k},  I_{k,n-k}) Q^{-1}) = \mathbb{V}_Q$.
\end{proof}

There is a well-known sequence of inclusions \cite[Proposition~11]{LL11} relating the complex, symplectic, and real Grassmannians:
\[
\frac{\U(n)}{\U(k) \times \U(n-k)} \subseteq \frac{\Sp(2n,\mathbb{R})}{\Sp(2k,\mathbb{R}) \times \Sp(2n-2k,\mathbb{R})} \subseteq \frac{\O(2n)}{\O(2k) \times \O(2n-2k)}.
\]
We will show this holds naturally with our involution models in \eqref{eq:invo} too.
\begin{corollary}\label{cor:SPidentification2}
Let $k, n \in \mathbb{N}$ with $k \le n$. Then we have
\[
\Gr(k, \mathbb{C}^n)  \xhookrightarrow{\psi_1} \Gr_{\Sp}(2k,\mathbb{R}^{2n}) \xhookrightarrow{\psi_2} \Gr(2k, \mathbb{R}^{2n})
\]
given by
\begin{align*}
\psi_1: \Gr(k, \mathbb{C}^{n})  &\to \Gr_{\Sp}(2k, \mathbb{R}^{2n}),\quad &\psi_1(A + \ii B) &= \begin{bsmallmatrix}
A & B \\
-B & A
\end{bsmallmatrix},  \\
\psi_2: \Gr_{\Sp}(2k, \mathbb{R}^{2n}) &\to \Gr(2k, \mathbb{R}^{2n}),\quad &\psi_2(X) &= 2 V_X V_X^\tp - I_{2n},
\end{align*}
where $A,  B\in \mathbb{R}^{n\times n}$ and $V_X \in \mathbb{R}^{2n \times 2k}$ is an orthonormal matrix whose column vectors form an orthonormal basis of $\mathbb{V}_X$.
\end{corollary}
\begin{proof}
We just need to verify that $\psi_1$ and $\psi_2$ are well-defined and injective. For $\psi_1$, these follow from the inclusion $\U(n) \subseteq \Sp(2n,\mathbb{R})$.  The symplectic subspace $\mathbb{V}_X$ is as defined in the proof of Corollary~\ref{cor:SPidentification1}. Clearly, the matrix $2 V_X V_X^\tp - I_{2n}$ does not depend on the choice of orthonormal basis $V_X$ of $\mathbb{V}_X$. Thus $\psi_2$ is well-defined.  That $\psi_2$ is injective follows from the observation that the involution $2 V_X V_X^\tp - I_{2n}$ is uniquely determined by $\mathbb{V}_X$.
\end{proof}

Now that we have established the legitimacy of our involution model \eqref{eq:GrSp} of the symplectic Grassmannian, we are ready to prove our result. We begin by showing that generically and up to a sign, every symplectic matrix is a product of two points on the symplectic Grassmannian.
\begin{lemma}\label{lem:generic}
If $X \in \Sp(2n,\mathbb{F})$ is generic,  then  $X = \pm Y_1 Y_2$ for some $Y_1, Y_2 \in \Gr_{\Sp}(2\lfloor n/2 \rfloor, \mathbb{F}^{2n})$.
\end{lemma}
\begin{proof}
We may assume without loss of generality that $X$ is diagonal. Since $X \in \Sp(2n,\mathbb{F})$, we must have $X = \diag(D,  D^{-1})$ for some $D  = \diag(d_1, \dots,  d_n) \in \mathbb{C}^{n\times n}$. Write $k \coloneqq \lfloor n/2 \rfloor$.  Then
\[
\begin{bmatrix}
D & 0 \\
0 & D^{-1}
\end{bmatrix} = \prod_{j=1}^k \begin{bmatrix}
D_j & 0 \\
0 & D_j^{-1}
\end{bmatrix} 
\]
where $D_j = \diag(\underbrace{1,\dots, 1}_{2j-2},  d_{2j-1},  d_{2j}, \underbrace{1, \dots, 1}_{2k - 2j})$, $j=1,\dots, k-1$, and 
\[
D_k = \begin{cases}
\diag(1,\dots, 1,  d_{2k-1},  d_{2k}) &\text{if } n = 2k, \\
\diag(1,\dots, 1,  d_{2k-1},  d_{2k},d_{2k+1}) &\text{if } n = 2k+1,
\end{cases}
\]
i.e., the first $2k-2$ diagonal entries of $D_k$ are ones.
Now for each $\diag(D_j,  D_j^{-1})$, there is a subspace $\mathbb{V}_j \subseteq \mathbb{F}^{2n}$ such that $\diag(D_j,  D_j^{-1})$ is symplectic when restricted to $\mathbb{V}_j$.  Moreover,  we may take $\mathbb{V}_j$ so that $\dim_{\mathbb{F}} \mathbb{V}_j = 4$ or $6$ by construction.  Hence  it suffices to prove the required result for $n = 2$ and $3$.  Since $X$ is clearly similar to $X^{-1}$,  $X$ is a product of two symplectic involutions by \cite[Theorem~8]{dR15} and \cite[Theorem~2]{Ad20}.  For $n = 2$,  we have $-\Gr_{\Sp}(0,\mathbb{F}^4) = \Gr_{\Sp}(4,\mathbb{F}^4) = \{I_{4}\}$.  Hence $X = Y_1 Y_2$ for some $Y_1,  Y_2 \in \Gr_{\Sp}(2,\mathbb{F}^4)$ unless  $X^2=I_4$.   
For $n = 3$,  we have $-\Gr_{\Sp}(0,\mathbb{F}^6) = \Gr_{\Sp}(6,\mathbb{F}^6) = \{ I_6 \}$; and we also have $-\Gr_{\Sp}(2,\mathbb{F}^6) =\Gr_{\Sp}(4,\mathbb{F}^6)$. Hence $X = \pm Y_1 Y_2$ for some $Y_1,Y_2 \in \Gr_{\Sp}(2,\mathbb{F}^6)$ 
unless $X^2 = I_6$.  Since $X$ is generic, none of these pathological cases apply and we must have $X = \pm Y_1 Y_2$ for some $Y_1,  Y_2 \in \Gr_{\Sp}(2\lfloor n/2 \rfloor, \mathbb{F}^{2n})$ whenever $n = 2$ or  $3$.
\end{proof}

We recall an almost self-evident observation about a topological group.
\begin{lemma}\label{lem:generation}
Let $G$ be a topological group and $U \subseteq G$ be an open dense subset.  If $U = U^{-1}$, then $G = U U$. 
\end{lemma}
\begin{proof}
Clearly $U^{-1}$ is an open dense subset of $G$. For each $g \in G$,  $g U^{-1}$ is also an open dense subset of $G$.  So there exists some $x \in gU^{-1} \cap U$, i.e., $x = g y$ for some $y \in U^{-1}$. Hence $g = x y^{-1} \in U U$.
\end{proof}

By combining Lemmas~\ref{lem:generic} and \ref{lem:generation}, we arrive at our required result.
\begin{theorem}[Symplectic group as product of four Grassmannians]\label{thm:GrSp}
Let $k \in \mathbb{N}$ and $2 \le n \in \mathbb{N}$.  Then 
\[
\Sp(2n,  \mathbb{F}) =
\begin{cases}
\Phi_{\Sp}(2k,2k,2k,2k,\mathbb{F}^{2n}) &\text{if } n  = 2k,\\
\Phi_{\Sp}(2k,2k,2k,2k,\mathbb{F}^{2n}) \medcup \Phi_{\Sp}(2k + 2,2k,2k,2k,\mathbb{F}^{2n}) &\text{if } n  = 2k + 1.
\end{cases}
\]
\end{theorem}
\begin{proof} 
Let $k \coloneqq \lfloor n/2 \rfloor$.  We claim that every $X \in \Sp(2n,\mathbb{F})$ can be decomposed as $X = \pm Y_1 Y_2 Y_3 Y_4$ for some $Y_1,Y_2, Y_3,  Y_4 \in \Gr_{\Sp}(2k,\mathbb{F}^{2n})$. The $\pm$ sign is inconsequential: For $n = 2k$, $Y\in  \Gr_{\Sp} (2k,\mathbb{F}^{2n})$ if and only if $-Y \in  \Gr_{\Sp} (2n - 2k,\mathbb{F}^{2n}) = \Gr_{\Sp} (2k,\mathbb{F}^{2n})$. For $n = 2k + 1$, $Y\in  \Gr_{\Sp} (2k,\mathbb{F}^{2n})$ if and only if $-Y \in  \Gr_{\Sp} (2n - 2k,\mathbb{F}^{2n}) = \Gr_{\Sp} (2k + 2,\mathbb{F}^{2n})$. It remains to prove the claim:  By Lemma~\ref{lem:generic},  there is an open dense $U \subseteq \Sp(2n,\mathbb{F})$ such that each $X \in U$ is a product of two elements in $\Gr_{\Sp}(2\lfloor n/2 \rfloor,\mathbb{F}^{2n})$ up to a sign.  By Lemma~\ref{lem:generation},  we have $UU = \Sp(2n,\mathbb{F})$.
\end{proof}

\section{Algorithms for Grassmannian decompositions}\label{sec:alg}

We will describe explicit algorithms for decomposing orthogonal and unitary matrices into products of points on the real and complex Grassmannians.
\begin{breakablealgorithm}
\caption{Decomposition of $\SO(n)$ into products of Grassmannians}
\label{alg:SO(n)}
\begin{algorithmic}[1]
\REQUIRE $Z \in \SO(n)$
\ENSURE $Z=X_1 X_2$ with $X_1, X_2 \in \Gr(\lfloor n/2 \rfloor, \mathbb{R}^{n})$
\STATE Compute real Schur decomposition
$Z = Q \diag(R_1, \dots, R_m) Q^\tp$, $Q \in \O(n)$, 
\[
R_k = \begin{cases}
\begin{bsmallmatrix}
    \cos 2\theta_k & \sin 2\theta_k \\
    -\sin 2\theta_k  & \cos 2\theta_k
\end{bsmallmatrix}   &\text{if~$k \in \{1,\dots, m-1\}$ or $k = m$ and $n$ is even}, \\
+1 &\text{if~$k=m$ and $n$ is odd}.
\end{cases}
\]
\STATE For $k =1,\dots, m$, compute
\begin{align*}
S_k &= \begin{cases}
\begin{bsmallmatrix}
    \cos\theta_k & -\sin \theta_k \\
    -\sin \theta_k & -\cos \theta_k
\end{bsmallmatrix}  &\text{if~$k \in \{1,\dots, m-1\}$ or $k = m$ and $n$ is even}, \\
-1  &\text{if~$k=m$ and $n$ is odd},
\end{cases} \\
T_k &= \begin{cases}
\begin{bsmallmatrix}
    \cos\theta_k & \sin \theta_k \\
    \sin \theta_k & -\cos \theta_k
\end{bsmallmatrix} &\text{if~$k \in \{1,\dots, m-1\}$ or $k = m$ and $n$ is even}, \\
-1  &\text{if~$k=m$ and $n$ is odd}.
\end{cases}
\end{align*}
\STATE Compute 
\[
X_1 = Q \diag(S_1,\dots,S_m)Q^\tp, \qquad X_2 = Q \diag(T_1,\dots,T_m)Q^\tp.
\]
\end{algorithmic}
\end{breakablealgorithm}
To obtain the corresponding algorithm for $\SO^\m (n)$, we simply take $R_m=-1$ and $S_m=1$ when $n$ is odd in Algorithm~\ref{alg:SO(n)}.

\begin{breakablealgorithm}
\caption{Decomposition of $\SU(n)$ into products of Grassmannians}
\label{alg:SU(n)}
\begin{algorithmic}[1]
\REQUIRE $Z \in \SU(n)$
\ENSURE $Z=X_1 X_2 X_3 X_4$ with $X_1, X_2,X_3,X_4 \in \Gr(\lfloor n/2 \rfloor, \mathbb{C}^{n})$
\STATE Compute eigenvalue decomposition $Z = Q \diag(e^{\ii\gamma_1},\dots,e^{\ii\gamma_n}) Q^\h$, $Q \in \U(n)$.
\STATE Compute $\gamma_n = -\sum_{j=1}^{n-1} \gamma_j$.
\STATE For $k=1,\dots, \lfloor n/2 \rfloor$,  compute $\alpha_k = \sum_{j=1}^{2k-1} \gamma_j$.
\STATE For $k=1,\dots, \lceil n/2 \rceil-1$,  compute $\beta_k = \sum_{j=1}^{2k} \gamma_j$.
\STATE Compute
\begin{align*}
X_1 &= \begin{cases}
Q \diag \bigl( \underbrace{K_2,\dots, K_2}_{ n/2} \bigr) Q^\h  &\text{if $n$ is even},  \\
Q \diag \bigl( \overbrace{K_2,\dots, K_2}^{ (n-1)/2},  -1\bigr) Q^\h &\text{otherwise},
\end{cases}
\\
X_2 &= \begin{cases}
Q \diag \Bigl( \begin{bsmallmatrix}
    0 & e^{-\ii\alpha_1} \\ e^{\ii\alpha_1} & 0
    \end{bsmallmatrix},  \dots,  \begin{bsmallmatrix}
    0 & e^{-\ii\alpha_{ n/2 }} \\ e^{\ii\alpha_{ n/2 }} & 0
    \end{bsmallmatrix} \Bigr) Q^\h     &\text{if $n$ is even},  \\[2ex]
    Q \diag \Bigl( \begin{bsmallmatrix}
    0 & e^{-\ii\alpha_1} \\ e^{\ii\alpha_1} & 0
    \end{bsmallmatrix},  \dots,  \begin{bsmallmatrix}
    0 & e^{-\ii\alpha_{ (n-1)/2 }} \\ e^{\ii\alpha_{ (n-1)/2}} & 0
    \end{bsmallmatrix},  -1 \Bigr) Q^\h   &\text{otherwise}, \\
\end{cases}    
    \\
X_3 &= 
\begin{cases}
Q \diag \bigl(-1,  \underbrace{K_2,\dots,  K_2}_{n/2 -1 },  1 \bigr) Q^\h   &\text{if $n$ is even},  \\
Q \diag \bigl(-1,  \overbrace{K_2,\dots,  K_2}^{ (n-1)/2 } \bigr) Q^\h   &\text{otherwise},  \\
\end{cases}  \\
X_4 &= 
\begin{cases}
Q \diag\Bigl(-1,  \begin{bsmallmatrix}
     0 & e^{-\ii\beta_1} \\ e^{\ii\beta_1} & 0
    \end{bsmallmatrix}, \dots,  \begin{bsmallmatrix}
     0 & e^{-\ii\beta_{n/2  - 1}} \\ e^{\ii\beta_{ n/2 - 1}} & 0
    \end{bsmallmatrix},  1 \Bigr) Q^\h   &\text{if $n$ is even},  \\[2ex]
Q \diag\Bigl(-1,  \begin{bsmallmatrix}
     0 & e^{-\ii\beta_1} \\ e^{\ii\beta_1} & 0
    \end{bsmallmatrix}, \dots,  \begin{bsmallmatrix}
     0 & e^{-\ii\beta_{ (n-1)/2 }} \\ e^{\ii\beta_{(n-1)/2}} & 0
    \end{bsmallmatrix} \Bigr) Q^\h  &\text{otherwise}.
    \end{cases} 
\end{align*}
\end{algorithmic}
\end{breakablealgorithm}
To obtain the corresponding algorithm for $\SU^\m (n)$, we simply take
\[
X_1 = Q \diag \bigl(\underbrace{K_2,\dots, K_2}_{(n-1)/2},  1\bigr) Q^\h
\]
when $n$ is odd, and 
\[
X_4 = Q \diag\Bigl(-1,  \begin{bsmallmatrix}
     0 & e^{-\ii\beta_1} \\ e^{\ii\beta_1} & 0
    \end{bsmallmatrix}, \dots,  \begin{bsmallmatrix}
     0 & e^{-\ii\beta_{n/2  - 1}} \\ e^{\ii\beta_{ n/2 - 1}} & 0
    \end{bsmallmatrix},  -1 \Bigr) Q^\h
\]
when $n$ is even in Algorithm~\ref{alg:SU(n)}.

We have one of our referees to thank for this section. These algorithms may be viewed as constructive proofs for the existence of the decompositions in Theorems~\ref{thm:prod-dim} and~\ref{thm:four}. Nevertheless they do not yield the dimension formulas in \eqref{eq:dim1} and Theorem~\ref{thm:prod_C_M2-dim}, which are ultimately the centerpiece of our nonconstructive approach.  We do not have a similar algorithm for the Grassmannian decomposition of symplectic matrices in Theorem~\ref{thm:GrSp} as the symplectic case relies on a very different argument.

\section{Conclusion}

In this article, we have shown that:
\begin{enumerate}[label={\upshape(\alph*)}, ref=\alph*] 
\item\label{it:SO} any $Q \in \SO(n)$ has $Q = X_1 X_2$ with $X_i \in \Gr(\lfloor n/2 \rfloor, \mathbb{R}^n)$;
\item\label{it:SU} any $Q \in \SU(n)$ has $Q = X_1 X_2 X_3 X_4$ with $X_i \in \Gr(\lfloor n/2 \rfloor, \mathbb{C}^n)$;
\item\label{it:Sp} any $Q \in \Sp(2n,\mathbb{F})$ has $Q = X_1 X_2 X_3 X_4$ with $X_i \in \Gr_{\Sp}(2\lfloor n/2 \rfloor, \mathbb{F}^{2n})$.
\end{enumerate}
We are optimistic that these hitherto unknown structures would have interesting implications and applications but they require further explorations beyond the scope of the present work and we will not speculate. Instead we will conclude our article with a discussion of one loose end and some open questions.

For \eqref{it:SO} and \eqref{it:SU} we have gone further than showing existence, and also have characterized these decompositions in the following sense: Theorem~\ref{thm:prod-dim} tells us the values of $k_1,k_2$ for which $\SO(n) = \Phi(k_1,k_2, \mathbb{R}^n)$ and likewise for $\SO^\m(n)$. Theorem~\ref{thm:four} does the same for $\SU(n)$ and $\SU^\m (n)$ with $\Phi(k_1,k_2,k_3,k_4,\mathbb{C}^n)$. 

For \eqref{it:Sp}, whether $\Sp(2n,\mathbb{F}) = \Phi_{\Sp}(k_1,k_2,k_3,k_4,\mathbb{F}^{2n})$ for values of $k_1,k_2,k_3,k_4$ other than the case treated in  Theorem~\ref{thm:GrSp} and those excluded by simple dimension counting\footnote{For example, as $\dim_{\mathbb{F}} \Gr_{\Sp}(2k,\mathbb{F}^{2n}) = 2k(2n - 2k)$ and $\dim \Sp(2n, \mathbb{F}) = 2n^2 + n$, if $k_1,k_2,k_3,k_4 \le n/8$, then $\Phi_{\Sp}(k_1,k_2,k_3,k_4,\mathbb{F}^{2n})$ has to be a proper subset of $\Sp(2n,\mathbb{F})$.} is wide open. We are also curious if this Grassmannian product structure extends to more general Lie groups, particularly $\SO(p,q)$, $\SU(p,q)$, $\SL(n)$, and $\Spin(n)$.  The compact symplectic group $\Sp(n) = \Sp(2n,\mathbb{C}) \cap \U(2n)$ deserves special highlight as one of our referees had asked about it. We pose this as an open problem:
\begin{open}\label{thm:opn}
Is there a decomposition of the compact symplectic group $\Sp(n)$ into a product of quaternionic Grassmannians?
\end{open}
Here the quaternionic Grassmannian is
\begin{align*}
\mathbb{G}(k, \mathbb{H}^n) &= \{k\text{-dimensional quaternionic subspaces in }\mathbb{H}^n \} \\
&\cong \Sp(n) /\bigl( \Sp(k) \times \Sp(n - k) \bigr).
\end{align*}
Our approaches in this article run into difficulties as $\mathbb{H}$ is a skew field. We leave these explorations and questions to future work and for interested readers.  

\subsection*{Acknowledgement} We are indebted to our two anonymous referees for their exceptionally helpful suggestions and insightful comments,  which have significantly improved our article.
\bibliographystyle{abbrv}

\end{document}